\documentclass[3p,fleqn]{elsarticle}

\usepackage{amssymb}
\usepackage{amsthm}
\usepackage[colorlinks=true, allcolors=blue]{hyperref}

\RequirePackage[mathlines]{lineno}[4.41]
\RequirePackage[leqno]{amsmath}[2.14]
\RequirePackage{xr-hyper}[6.00]
\RequirePackage{ifpdf}[2.3]
\if@bakoma
\else
\RequirePackage{xcolor}[2.11]
\RequirePackage[all]{hypcap} [1.11]
\RequirePackage{algorithm}[0.1]
\RequirePackage[capitalize,nameinlink]{cleveref}[0.19]

\usepackage{graphicx} 
\usepackage{xspace}
\usepackage{multirow		}
\usepackage{bbm}
\usepackage{Macros/mydef}
\usepackage{Macros/remarks}
\usepackage{algorithm}
\usepackage[algo2e]{algorithm2e} 
\usepackage{algpseudocode}
\usepackage[numbers]{natbib}
\usepackage{graphicx}
\usepackage{subcaption}
\usepackage{enumitem}
\usepackage{mathtools}
\usepackage{xcolor}

\usepackage[utf8]{inputenc}
\usepackage[english]{babel}
\usepackage[colorinlistoftodos]{todonotes}
\usepackage[utf8]{inputenc}
\usepackage[export]{adjustbox}
\usepackage{expl3}[2012-07-08]
\ExplSyntaxOn
\cs_new_eq:NN \fpeval \fp_eval:n
\ExplSyntaxOff

\theoremstyle{plain}

\numberwithin{theorem}{section}

\numberwithin{lemma}{section}

\newtheorem{proposition}{Proposition}
\numberwithin{proposition}{section}

\theoremstyle{definition}
\newtheorem{remark}{Remark}
\numberwithin{remark}{section}

\journal{Journal of Computational and Applied Mathematics}

\let\cite=\citet
\newcommand{\commentB}[1]{\textcolor{black}{#1}}
\newcommand{\commentA}[1]{\textcolor{black}{#1}}

\begin{document}

\begin{frontmatter}

\title{A high-order artificial compressibility method based on Taylor series time-stepping for variable density flow }

\author[UUIT]{Lukas Lundgren\corref{cor1}}
\ead{lukas.lundgren@it.uu.se}

\author[UUIT]{Murtazo Nazarov}
\ead{murtazo.nazarov@it.uu.se}

\cortext[cor1]{Corresponding author}

\address[UUIT]{Uppsala University, Department of Information Technology, P O Box 337, S-751 05 Uppsala, Sweden}



\begin{abstract}
In this paper, we introduce a fourth-order accurate finite element method for incompressible variable density flow. The method is implicit in time and constructed with the Taylor series technique, and uses standard high-order Lagrange basis functions in space. Taylor series time-stepping relies on time derivative correction terms to achieve high-order accuracy. We provide detailed algorithms to approximate the time derivatives of the variable density Navier-Stokes equations. Numerical validations confirm a fourth-order accuracy for smooth problems. We also numerically illustrate that the Taylor series method is unsuitable for problems where regularity is lost by solving the 2D Rayleigh-Taylor instability problem. 
\end{abstract}


\end{frontmatter}


\markboth{AUTHORS}{Lukas Lundgren and Murtazo Nazarov}


\section{Introduction} 
The simulation of incompressible variable density flow plays an important role in several areas of fluid dynamics. Its importance stems from its usefulness when simulating flow largely affected by density variations. This situation occurs in many places in nature, such as stratified flow in the ocean and the mixing of fluids with distinct phases, e.g. oil and water. The governing equations that we consider in this manuscript are the incompressible Navier-Stokes equations, augmented with an advection equation for the density. 


In the finite element literature concerning incompressible flow, the divergence-free constraint is often imposed using a projection method. For an overview of projection methods, we refer to the review paper by \cite{Guermond_Minev_2006} and the references therein. There are two main disadvantages of using a projection method: the order of accuracy (in time) is limited to second-order and each time-step involves solving a linear system for the pressure which scales with $\mathcal{O}(h^{-2})$, where $h$ is the grid-spacing and may become a bottleneck for small mesh sizes.

Another approach, to impose the divergence-free constraint, is to use an artificial compressibility method. Many variations of these types of methods have been proposed over the years, the first of which was proposed by \cite{Chorin1967} and \cite{Temam_1968} in the 1960s and has been further developed by \cite{Shen_1996} and others. Recent examples of artificial compressibility methods in many different forms include \citep{DeCaria_2017, Yang_2016, Guermond_Minev_2015, Guermond_Minev_2017, Guermond_Minev_2019, Layton_2020, Chen_Layton_2019, DeCaria_2019, DeCaria_2020,Cox_2016,milani_thesis}. In particular, \cite{Guermond_Minev_2015} have recently proposed an artificial compressibility method that uses a defect-correction time-stepping scheme to achieve third-order accuracy and were able to prove unconditional stability (for the unsteady Stokes-equations).  The method involves solving a linear system which scales with $\mathcal{O}(\tau h^{-2})$, where $\tau$ is the time-step, which is an improvement over $\mathcal{O}(h^{-2})$. The authors have further developed the method to decrease computational complexity \citep{Guermond_Minev_2017} and added a robust time-step control \citep{Guermond_Minev_2019}. Currently, the method proposed by \cite{Guermond_Minev_2019} is limited to constant density, and the extension to variable density is still lacking in the literature. The finite element methods proposed for variable density incompressible flow available in the literature are still mostly based on projection methods and are, as such, limited to second-order accuracy, see \citep{Guermond_Salgado_2009,Guermond_Salgado_2011,
Wu2017,Pyo_2007,Chen_2020} and the references therein. Another approach to achieve high-order accuracy in time is to solve the classical saddle-point system, which is roughly as fast as a projection method is \citep{Axelsson_2015}.

The main aim of this paper is to introduce a new fourth-order accurate finite element method for variable density flow. To achieve this, we have utilized the artificial compressibility method proposed by \cite{Guermond_Minev_2019}, and further developed their method to fit in the variable density context. We emphasize that this extension is not straightforward and there are many different approaches to tackle this problem, see the dissertation by \cite{Fahad_dissertation} and recent works by \cite{Bassi_2018} and \cite{Manzanero_2020} in the context of discontinuous Galerkin \commentA{methods}. More specifically, \cite{Bassi_2018} and \cite{Manzanero_2020} presented a high-order method where the time-stepping was performed explicitly which could lead to stiffness problems if the artificial compressibility penalty parameter is chosen too greedily.

This is in contrast to \cite{Guermond_Minev_2019}, who used the Taylor series method as an implicit time-stepping method. The time-stepping method relies on approximating time derivatives which are then used as correction terms to achieve high-order accuracy. Despite its good stability properties and high accuracy for smooth problems, it is currently not known if the method is suitable for nonsmooth problems where the regularity of the solution is lost. One of the goals of this article is to investigate this question by solving the so-called Rayleigh-Taylor instability problem. In this test case, the density has a discontinuity, and therefore the time derivative of the density becomes unbounded. This loss of regularity makes the Taylor series method unsuitable for problems with strong discontinuities.

This paper is organized as follows: In Section \ref{Sec:prelim} the governing equations are introduced and a recap of the first-order method by \cite{Guermond_Minev_2019} is given. In Section \ref{Sec:fem_approx} the finite-element discretization for the first-order method is presented. In Section \ref{Sec:time-stepping-method} the time-stepping method is derived which is the main contribution of this manuscript. In Section \ref{Sec:boundary_conditions} we give a brief discussion of the boundary conditions. In Section \ref{Sec:computations} we test our method against common benchmarks in the literature. In Section \ref{Sec:conclusion} we give concluding remarks.

\section{Preliminaries}\label{Sec:prelim}
In this section, we introduce the governing equations that model variable-density flow. We also give a brief overview of the artificial compressibility method to impose the divergence-free constraint. 

\subsection{Governing equations}\label{Sec:Equation}
We consider the incompressible Navier-Stokes equation with variable density in a domain $\Omega \subset \mb R^d$ and finite time interval $[0,T]$
\begin{equation}\label{eq:cons_law_primitive}
      \begin{aligned}
        \p_t \rho + \bu \cdot \nabla \rho  & = 0,\\
\rho (\p_t \bu + \bu \cdot \nabla  \bu ) + \nabla p - \mu \Delta \bu &= \bef, \quad &(\bx,t) \in  \Omega \times (0,T],\\
           \nabla \cdot \bu &= 0,\\
           \bu(\bx,0) &= \bu^0(\bx), \\
           \rho(\bx,0) &= \rho^0(\bx), \quad  &\bx\in \Omega,    
      \end{aligned}
\end{equation}
where the density $\rho > 0$, the velocity field $\bu$ and the pressure $p$ are the unknowns. $\bef(\bx,t)$ represents an external force, $\mu > 0$ is the dynamic viscosity and $\rho^0(\bx)$, $\bu^0(\bx)$ are initial conditions for the density and velocity. We assume that the governing equations are supplied with well-posed boundary conditions.

\subsection{Artificial compressibility} \label{Sec:artificial_compressibility}
In this section, we discuss the technique which we use to impose the divergence-free condition. To make the argument more general and in line with \citep{Guermond_Minev_2019}, we omit the spatial discretization for now. Later we provide the spatial discretization using continuous finite elements in Section \ref{Sec:fem_approx}. As a consequence, the scheme presented in this section applies to other spatial discretizations such as finite volume, finite differences, discontinuous Galerkin and so on. The main idea with artificial compressibility is to regularize the incompressibility constraint using a penalty parameter $\epsilon>0$ such that
\begin{equation} \label{eq:ac_constraint}
\epsilon p_t + \nabla \cdot \bu = 0,
\end{equation}
to impose the divergence-free condition weakly. One of the biggest benefits of using artificial compressibility to impose the divergence-free constraint, as opposed to using a projection method, is that it allows for high-order accuracy \citep{Guermond_Minev_2015}. We begin by recapping the first-order artificial compressibility method proposed in \citep{Guermond_Minev_2015, Guermond_Minev_2019} for constant density. Let $\tau = t^{n+1} - t^n $ be the current time-step and $( \bu^n, p^n)$ be solutions obtained at the current time $t^n$, then the procedure is as follows: given $\bu^n$ and $p^{n}$, solve for $\bu^{n+1}$ and $p^{n+1}$ as follows:
\begin{align}  \label{eq:ac:const_dens}
 \frac{\bu^{n+1} - \bu^n}{\tau} + \bu^{n} \cdot \nabla \bu^{n}   -  \lambda \nabla \nabla \cdot \bu^{n+1} + \nabla p^n
&=  \bef^{n+1} + \mu \Delta \bu^{n+1}, \\ 
p^{n+1} = p^n - \lambda \nabla \cdot \bu^{n+1}, \label{eq:ac:const_dens2}
\end{align}
where $\lambda := \tau/\epsilon$ is a penalty parameter. Here, another benefit becomes apparent, since the linear system constructed by \eqref{eq:ac:const_dens} scales with $\mathcal{O}(\tau h^{-2})$. This is an improvement over projection methods which requires solving a linear system for the pressure which scales with $\mathcal{O}(h^{-2})$, typically referred to as a pressure Poisson equation. This may become a bottleneck for small mesh sizes.

The main goal of this manuscript is to extend the method proposed in \citep{Guermond_Minev_2019} to variable density flows. We propose the following extension of \eqref{eq:ac:const_dens}-\eqref{eq:ac:const_dens2} to variable density flow: given $\rho^n$, $\bu^n$ and $p^{n}$, solve for $\rho^{n+1}$ $\bu^{n+1}$ and $p^{n+1}$ as follows:
\begin{align}  \label{eq:first_order1}
\frac{\rho^{n+1} - \rho^n}{\Delta t} + \bu^{n} \cdot \nabla \rho^{n+1} &= 0, \\ \label{eq:first_order2}
\rho^{n+1} \left( \frac{\bu^{n+1} - \bu^n}{\Delta t} + \bu^{n} \cdot \nabla \bu^{n} \right)  -  \lambda \nabla \nabla \cdot \bu^{n+1} + \nabla p^n 
&=  \bef^{n+1} + \mu \Delta \bu^{n+1}, \\ \label{eq:first_order3}
p^{n+1} = p^n - \lambda \nabla \cdot \bu^{n+1}.
\end{align}
The high-order accurate extension of \eqref{eq:first_order1}-\eqref{eq:first_order3} is derived in Section \ref{Sec:time-stepping-method} using the Taylor series method following the procedure outlined in \citep{Guermond_Minev_2019}.


\begin{remark}
To make the extension to variable density to be more in line with \commentA{\cite{Guermond_Minev_2019}}, we chose the advective form of the flux in \eqref{eq:first_order1} and \eqref{eq:first_order2}. To the best of the authors' knowledge, some modifications of the scheme are necessary to achieve an $L_2$-estimate. In Appendix \ref{Sec:appendix}, we present an $L_2$-estimate for the modified scheme. One requirement is that the flux is discretized implicitly or semi-implicitly and also that it is discretized such that energy conservation is possible. For an excellent comparison of different alternatives to discretize the flux, we refer to \cite{Charnyi2017}. Lastly, the time level of the density in the time derivative term needs to be shifted. The shifted density will, unfortunately, make high-order extensions (like those presented in this work) harder since it will limit the order of accuracy in time to first-order accuracy. There have been some attempts at overcoming this issue for BDF2 based time-stepping by \cite{Guermond_Salgado_2011}, but it remains an open question.

\end{remark}

\begin{remark}
Note that the term $-  \lambda \nabla \nabla \cdot \bu^{n+1} + \nabla p^n$ is, at this stage, equivalent to $\nabla p^{n+1}$. After spatial discretization the terms could be different.
\end{remark}

\section{Finite element approximations} \label{Sec:fem_approx} 
In this section we provide the spatial discretization of the first-order algorithm \eqref{eq:first_order1}-\eqref{eq:first_order3}. This is done using continuous finite elements. We denote the computational mesh by $\mathcal{T}_h$ which is a triangulation of $\Omega$ into a finite number of disjoint elements $K$. The finite element spaces we use for the density, velocity and pressure are respectively given by
\begin{equation}
\begin{split}
&M_h : = \{ w_h : w_h \in \mathcal{C}^0( \Omega ); \forall K \in \mathcal{T}_h, w_h \in \polP_k \}, \\ 
&\bV_h := [M_h]^d, \\
&Q_h : = \l \{ q_h : q_h \in \mathcal{C}^0( \Omega ); \forall K \in \mathcal{T}_h, q_h \in \polP_{k^*}, \int_\Omega q_h d \bx = 0 \r \}, 
\end{split}
\end{equation}
where $\polP_k$ and $\polP_{k^*}$ are the set of multivariate polynomials of total degree at most $k \geq 1$ and $k^* \geq 1$ defined over $K$. It is well-known that to satisfy the so-called inf-sup condition we require $k > k^*$. We often use the inner product
\begin{equation}
(v, w) := \sum_{K \in \mathcal{T}_h} \int_K v \cdot w \ d \bx.
\end{equation}

The finite element method discretization of \eqref{eq:first_order1}-\eqref{eq:first_order3} proceeds as follows: Given $(\rho^n_h, \bu^n_h, p_h^n) \in M_h \times V_h \times Q_h$ solve for $\rho^{n+1}_h \in M_h$ such that

\begin{equation} \label{eq:fem_cont_update}
\begin{split}
\frac{1}{\tau} \l( \rho^{n+1}_h, w_h \r) + \l( \bu^{n}_h \cdot \nabla \rho^{n+1}_h, w_h \r) + \sigma_h \l( \nabla \rho_h^{n+1}, \nabla w_h \r)  = \frac{1}{\tau} \l( \rho^{n}_h, w_h \r) , 
\end{split}\quad \forall w_h \in M_h,
\end{equation}
then solve for $\bu_h^{n+1} \in \bV_h$ such that 
\begin{equation} \label{eq:fem_mom_update}
\begin{split}
\frac{1}{\tau} \l( \rho^{n+1}_h \bu^{n+1}_h , \bv_h \r) + \lambda \l( \nabla \cdot \bu^{n+1}_h , \nabla \cdot \bv_h \r)      + \mu (\nabla \bu^{n+1}_h, \nabla \bv_h) \\ + \nu_h ( \rho_h^{n+1} \nabla \bu^{n+1}_h, \nabla \bv_h) =  \frac{1}{\tau} \l( \rho^{n+1}_h \bu^{n}_h , \bv_h \r) + ( \bef^{n+1}  - \rho_h^{n+1} \bu^{n}_h \cdot \nabla \bu^{n}_h  -  \nabla p^n_h, \bv_h), 
\end{split} \quad \forall \bv_h \in \bV_h,
\end{equation} 
where $\nu_h$ and $\sigma_h$ are artificial viscosity coefficients yet to be defined. Lastly solve for $p^{n+1}_h \in Q_h$ such that
\begin{equation} \label{eq:fem_press_update}
(p^{n+1}_h, q_h) = (p^n_h, q_h) - \lambda (\nabla \cdot \bu^{n+1}_h, q_h), \quad \forall q_h \in Q_h.
\end{equation}
Note in particular the added stabilization terms $\sigma_h \l( \nabla \rho_h^{n+1}, \nabla w_h \r)$ in the continuity update \eqref{eq:fem_cont_update} and $\nu_h ( \rho_h^{n+1} \nabla \bu^{n+1}_h, \nabla \bv_h)$ in the velocity update. These contributions are the added artificial viscosities which will help stabilize the finite element discretization.

Another important term in the finite element discretization is $\lambda \l( \nabla \cdot \bu^{n+1}_h , \nabla \cdot \bv_h \r)$. This term is commonly referred to as grad-div stabilization in the literature. Considerable attention has been spent on analyzing this term as an addition to finite element approximations of the Stokes equations \citep{Jenkins_2014,Olshanskii_2004}. Added grad-div stabilization to finite element approximations of the Navier-Stokes equations (with constant density) has also been investigated \citep{Olshanskii2009,Case_2011,John2010, Rohe2010, Frutos2018}. In particular, we mention the work by \cite{Case_2011} who presented convincing theoretical and numerical arguments for the case that increasing $\lambda$ will make sure that the divergence-free condition is satisfied more strongly.

The integration necessary for the finite element approximations can be computed exactly by using an appropriate quadrature rule. In practice, this can be done using finite element software and in this work, we use FEniCS \citep{fenics2015} for all computations. In this work, the integrals are therefore computed exactly.

\begin{remark}
It is well-known that for advection-dominated problems, mass-lumping introduces unfavorable dispersion errors \citep{Guermond_Pasquetti_2013}. However, as mentioned by \cite[Sec 4.3]{Guermond_Minev_2019}, there is no need to use the full mass matrix in the pressure update \eqref{eq:fem_press_update}. In \eqref{eq:fem_press_update}, mass-lumping won't decrease the accuracy significantly or change the properties of the scheme.
\end{remark}

\begin{remark} 
The grad-div operator $\lambda \l( \nabla \cdot \bu^{n+1}_h , \nabla \cdot \bv_h \r)$ couples the different components of the velocities which decreases sparsity of the linear system and makes preconditioning techniques more difficult. Moreover, as $\lambda$ increases, the condition number of the linear system
increases significantly, which makes the system increasingly more difficult to solve using an iterative method. One alternative to handle this is to use an approximation of the grad-div term \citep{Guermond_Minev_2017,Guermond_Minev_2019,Bowers2014,Linke_2013,Minev2018}. Another approach would be to use $\nabla p^{n+1}$ instead of $-  \lambda \nabla \nabla \cdot \bu^{n+1} + \nabla p^n$ in \eqref{eq:first_order2} and base the finite element approximation on that instead. This will lead to a saddle-point structure of the resulting linear systems, making it suitable for preconditioning using a Schur complement approach \citep{Kronbichler_2018,Axelsson_2015}. Since the focus of this study is spatial and temporal discretization, we have not pursued this further.
\end{remark}

\begin{remark}
The boundary contribution from the Laplacian term in \eqref{eq:fem_mom_update} has been omitted since the boundary conditions we consider in this manuscript make it so, see Section \ref{Sec:boundary_conditions} for more details.
\end{remark}

\section{High-order time-stepping}\label{Sec:time-stepping-method}
The goal of this section is to extend the high-order time-stepping scheme proposed in \citep{Guermond_Minev_2019} to variable density flow and fourth-order accuracy. This is also the main aim of this manuscript. The time-stepping method is the so-called Taylor series method which is a recursive algorithm that solves for the time derivatives of the solution in \textit{descending} order. From now on, let the subscript of a function denote the partial derivatives with respect to $t$, i.e.
\begin{equation}
u_l := \partial_t^l u.
\end{equation}
Additionally, we denote the superscript to be the time level. As an example, we have $u_3^n := \partial_{ttt} u(t^n)$. Also recall that $\tau = t^{n+1} - t^n $ is the current time-step.

To give a brief explanation of the Taylor series method, consider the nonlinear ODE system $u_t = f(u,t)$, where $f$ is assumed to be a sufficiently smooth function. The $p$-th order accurate Taylor series method then proceeds as follows: Given $u^n_{p-1}, u^n_{p-2}, \dots, u_{1}^{n}, u_{0}^{n}$, solve for $u^{n+1}_{p-1}, u^{n+1}_{p-2}, \dots, u_{1}^{n+1}, u_{0}^{n+1}$ as follows:
\begin{itemize}
\item Step 1: Solve for $u_{p-1}^{n+1}  + \mathcal{O}(\tau)$
\item Step 2: Solve for $u_{p-2}^{n+1}  + \mathcal{O}(\tau^2)$
\item[] $\vdots$
\item   Step $p-1$: Solve for $u_{1}^{n+1} + \mathcal{O}(\tau^{p-1})$
\item Step $p$: Solve for $u_0^{n+1} + \mathcal{O}(\tau^p)$
\end{itemize}

The accuracy of each steps 2 to $p$ is achieved by using correction terms based on Taylor series and the time derivatives from the previously computed steps. There is a lot of flexibility when using this technique since the above steps can be done fully explicitly, fully implicitly or using an implicit-explicit approach. The latter was done in \citep{Guermond_Minev_2019} to construct a time-stepping scheme for the constant density incompressible Navier-Stokes equations. Since the Taylor series method is a one-step method it is suitable for time-step control which was also demonstrated in \citep{Guermond_Minev_2019}. One major downside with the Taylor series method is that it requires sufficient smoothness of the solution since the method relies on that the existence of time derivatives. In Section \ref{Sec:rt_ins} we test the Taylor series time-stepping method against a problem where the regularity of the solution is lost.

\subsection{Third-order accurate time-stepping for constant density}
The high-order time-stepping by \cite{Guermond_Minev_2019} is built upon the first-order time-stepping scheme \eqref{eq:ac:const_dens}-\eqref{eq:ac:const_dens2}. To simplify the notation we define $\bN( \bu ) := \bu \cdot \nabla \bu$ as the nonlinear operator of the Navier-Stokes equations. Also let $\nu$ be the physical kinematic viscosity. The third order Taylor series method \citep{Guermond_Minev_2019} proceeds as follows: Given $\bu_0^n, \bu_1^n, \bu_2^n$ and $p_0^n, p_1^n, p_2^n$, compute $\bu_0^{n+1}, \bu_1^{n+1}, \bu_2^{n+1}$ and $p_0^{n+1}$, $p_1^{n+1}$, $p_2^{n+1}$:
\begin{align}
\begin{split}
\frac{\bu_2^{n+1}}{\tau} - \nu \Delta \bu_2^{n+1} - \lambda \nabla \nabla \cdot \bu_2^{n+1} = \frac{\bu_2^{n}}{\tau} + \bef^{n+1}_2 - \nabla p_2^n \\
- \frac{\bN \l(\bu_0^n + 2 \tau \bu_1^n + 2 \tau^2 \bu_2^n \r) - 2 \bN \l(\bu_0^n + \tau \bu_1^n + \frac{\tau^2}{2} \bu_2^n \r) + \bN(\bu_0^n)}{\tau^2},
\end{split} \label{eq:GM_stage1} \\ \nonumber 
\\
\begin{split}p_2^{n+1} = p_2^n - \lambda \nabla \cdot  \bu^{n+1}_2, \end{split} \\ 
\nonumber \\
\begin{split}
\frac{\bu_1^{n+1}}{\tau} - \nu \Delta \bu_1^{n+1} - \lambda \nabla \nabla \cdot  \bu_1^{n+1} = \frac{\bu_1^{n}}{\tau} + \bef^{n+1}_1  - \nabla \l( p_1^n + \tau p_2^{n+1} \r) \\
-\frac{\tau}{2}\frac{\bu_2^{n+1} - \bu_2^n}{\tau}- \frac{\bN \l(\bu_0^n + 2 \tau \bu_1^n + 2 \tau^2 \bu_2^n \r) -  \bN(\bu_0^n)}{2 \tau},
\end{split}\label{eq:GM_stage2}  \\ \nonumber 
\\
\begin{split}p_1^{n+1} = p_1^n + \tau p_2^{n+1}    - \lambda \nabla \cdot \bu^{n+1}_1, \end{split} \\ \nonumber 
\\
\begin{split}
\frac{\bu_0^{n+1}}{\tau} - \nu \Delta \bu_0^{n+1} - \lambda \nabla \nabla \cdot  \bu_0^{n+1} = \frac{\bu_0^{n}}{\tau} + \bef^{n+1}_0 - \nabla \l( p_0^n + \tau p_1^{n+1} + \frac{\tau^2}{2} p_2^{n+1} \r) \\
-\frac{\tau}{2}\frac{\bu_1^{n+1} - \bu_1^n}{\tau} - \frac{\tau^2}{12} \frac{\bu_2^{n+1} - \bu_2^n}{\tau} - \bN \l(\bu_0^n + \tau \bu_1^n + \frac{\tau^2}{2} \bu_2^n \r),
\end{split} \label{eq:GM_stage3}  \\ \nonumber \\
\begin{split} p_0^{n+1} = p_0^n + \tau p_1^{n+1} - \frac{\tau^2}{2} p_2^{n+1} - \lambda \nabla \cdot  \bu^{n+1}_0, \end{split} \label{eq:GM_press_stage3}
\end{align}

\begin{proposition} 
The algorithm \eqref{eq:GM_stage1}-\eqref{eq:GM_press_stage3}, is unconditionally stable if $\bN(\bu) = 0$ and third-order accurate. \citep[Sec 3.2]{Guermond_Minev_2019}
\end{proposition}

The nonlinear terms $\bN( \bu)$ in \eqref{eq:GM_stage1}-\eqref{eq:GM_press_stage3} are discretized explicitly to avoid solving an expensive nonlinear system. The time derivatives of the nonlinear terms are discretized using finite differences.

\begin{remark}
One of the novelties of \citep{Guermond_Minev_2015, Guermond_Minev_2019} is the technique used to bootstrap the pressure in \eqref{eq:GM_stage1}-\eqref{eq:GM_press_stage3}. The benefit of the proposed time-stepping method is that the divergence-free constraint is high-order accurate for $\lambda = 1$. As mentioned by \cite{Chen_Layton_2019}, another alternative to impose the divergence-free constraint more strongly is to work directly with the first-order method \eqref{eq:ac:const_dens}-\eqref{eq:ac:const_dens2}. By increasing $\lambda$, the divergence-free constraint is imposed more strongly, at the cost of a higher condition number of the resulting linear systems.
\end{remark}


\subsection{Fourth-order accurate time-stepping for variable density}
We now turn our attention to the main focus of this study, which is to follow the procedure outlined in \citep{Guermond_Minev_2019} to derive a corresponding fourth-order accurate time-stepping scheme for variable density flow. Repeated time differentiation of the continuity equation in the incompressible variable density Navier-Stokes equations \eqref{eq:cons_law_primitive} yields
\begin{align} \label{eq:rho_first}
\partial_t \rho_3 + \bu_0 \cdot \nabla \rho_3 + 3 \bu_1 \cdot \nabla \rho_2 + 3 \bu_2 \cdot \nabla \rho_1 + \bu_3 \cdot \nabla \rho_0  = 0, \\
\partial_t \rho_2 + \bu_0 \cdot \nabla \rho_2 + 2 \bu_1 \cdot \nabla \rho_1 + \bu_2 \cdot \nabla \rho_0   = 0,\label{eq:rho_2} \\
\partial_t \rho_1  + \bu_0 \cdot \nabla \rho_1 + \bu_1 \cdot \nabla \rho_0  = 0, \label{eq:rho_1}\\
\partial_t \rho_0 + \bu_0 \cdot \nabla \rho_0 = 0. \label{eq:rho_last}
\end{align}
Given $\bu_0^n, \bu_1^n, \bu_2^n, \bu_3^n$ and $\rho_0^n, \rho_1^n, \rho_2^n, \rho_3^n $, the idea is then to compute $\rho_0^{n+1}, \rho_1^{n+1}, \rho_2^{n+1}$ and $\rho_3^{n+1}$ using \eqref{eq:rho_first}-\eqref{eq:rho_last}. The full algorithm of this is presented in Section \ref{Sec:high_order_cont_update}. Then $\rho_0^{n+1}, \rho_1^{n+1}, \rho_2^{n+1}$, $\rho_3^{n+1}$ and $\bu_0^n$, $\bu_1^n$, $\bu_2^n$, $\bu_3^n$ are used to compute $\bu_0^{n+1}$, $\bu_1^{n+1}$, $\bu_2^{n+1}$, $\bu_3^{n+1}$ and $p_0^{n+1}$, $p_1^{n+1}$, $p_2^{n+1}$, $p_3^{n+1}$ via a similar velocity update as \eqref{eq:GM_stage1}-\eqref{eq:GM_press_stage3}. Again, we derive our velocity update by repeated time differentiation of the remaining governing equations in \eqref{eq:cons_law_primitive} with careful application of the product rule
\begin{eqnarray} \label{eq:mom_differentiation_3}
&\begin{split}
 \rho_0  \l( \partial_t \bu_3 + \partial_{ttt} \bN(\bu_0) \r) + 3 \rho_1 \l( \bu_3 + \partial_{tt} \bN(\bu_0) \r) \\ + 3 \rho_2 \l( \bu_2 + \partial_{t} \bN(\bu_0) \r) + \rho_3 \l( \bu_1 + \bN(\bu_0) \r)   - \mu \Delta \bu_3 +  \nabla p_3 = \bef_3(t), \end{split} \quad & \nabla \cdot \bu_3 = 0, \\ \nonumber \\ \label{eq:mom_differentiation_2}
&\begin{split}
 \rho_0  \l( \partial_t \bu_2 + \partial_{tt} \bN(\bu_0) \r) + 2 \rho_1 \l( \bu_2 + \partial_{t} \bN(\bu_0) \r) \\ + \rho_2 \l( \bu_1 + \bN(\bu_0) \r)  - \mu \Delta \bu_2 +  \nabla p_2 = \bef_2(t), \end{split} \quad & \nabla \cdot \bu_2 = 0, \\ \nonumber \\ \label{eq:mom_differentiation_1}
&\begin{split}
 \rho_0  \l( \partial_t \bu_1 + \partial_t \bN(\bu_0) \r) + \rho_1 \l( \bu_1 + \bN(\bu_0) \r)  - \mu \Delta \bu_1 +  \nabla p_1 = \bef_1(t), \end{split} \quad & \nabla \cdot \bu_1 = 0, \\ \nonumber \\ \label{eq:mom_differentiation_0}
&\rho_0 \l( \partial_t \bu_0 + \bN(\bu_0) \r) - \mu \Delta \bu_0 + \nabla p_0 = \bef_0(t), \quad & \nabla \cdot \bu_0 = 0.
\end{eqnarray}

To handle the divergence-free constraint in \eqref{eq:mom_differentiation_3}-\eqref{eq:mom_differentiation_0} the artificial compressibility method will be utilized. When using this method, bootstrapping of the pressure is necessary for high-order extensions and, here, we follow exactly the method used in the constant density case \eqref{eq:GM_stage1}-\eqref{eq:GM_press_stage3}. Similarly, the nonlinear terms $\bN( \bu)$ are discretized explicitly and the time derivatives of these are discretized using finite differences.

We now introduce some notation which will be helpful when writing out the scheme. Frequently, extrapolated solutions of time level $n+s$, where $s$ is an integer, will be used. These are denoted as $(\tilde{\bu}_0^{n+s}, \tilde{\rho}_0^{n+s})$ and are given by the following extrapolation formulas
\begin{align}
\begin{split}
\tilde{\bu}_0^{n+s} = \bu_0^n + s\tau \bu_1^n + \frac{(s \tau)^2}{2} \bu_2^n +  \frac{(s\tau)^3}{6 } \bu_3^n, \\
 \tilde{\rho}_0^{n+s} = \rho_0^n + s \tau \rho_1^n + \frac{(s \tau)^2}{2} \rho_2^n +  \frac{(s \tau)^3}{6 } \rho_3^n.
\end{split}
\end{align}
Similarly, extrapolated solutions of the time derivatives of time level $n+s$ are given below
\begin{align}
\begin{split}
\tilde{\bu}_1^{n+s} = \bu_1^n + s\tau \bu_2^n + \frac{(s \tau)^2}{2} \bu_3^n, \quad \tilde{\bu}_2^{n+s} = \bu_2^n + s\tau \bu_3^n, \\
 \tilde{\rho}_1^{n+s} = \rho_1^n + s \tau \rho_2^n + \frac{(s \tau)^2}{2} \rho_3^n, \quad \tilde{\rho}_2^{n+s} = \rho_2^n + s \tau \rho_3^n.
\end{split}
\end{align}

\subsubsection{Continuity update} \label{Sec:high_order_cont_update}
In this section we present the time-stepping algorithm for the continuity update. The aim is to discretize \eqref{eq:rho_first}-\eqref{eq:rho_last} using the Taylor series technique. We propose the following discretization of \eqref{eq:rho_first}: Given $\bu_0^n, \bu_1^n, \bu_2^n, \bu_3^n$ and $\rho_0^n, \rho_1^n, \rho_2^n, \rho_3^n$, compute $\rho_3^{n+1}$ as follows:
\begin{equation} \label{eq:3_con_update}
\begin{split}
\frac{\rho_3^{n+1} - \rho_3^{n}}{\tau}  + \tilde{\bu}_0^{n+1} \cdot \nabla \rho_3^{n+1}  + 3 \tilde{\bu}_1^{n+1} \cdot \nabla \tilde{\rho}_2^{n+1} + 3 \tilde{\bu}_2^{n+1} \cdot \nabla \tilde{\rho}_1^{n+1}   + \tilde{\bu}_3^{n+1} \cdot \nabla \tilde{\rho}_0^{n+1} =  \sigma_{h,3} \Delta \rho_3^{n+1},
\end{split}
\end{equation}
where the term $ \sigma_{h,3} \Delta \rho_3^{n+1}$ is an artificial viscosity term which we use to stabilize $\rho_3^{n+1}$ in space. The update \eqref{eq:3_con_update} leads to a first-order (in time) approximation of $\rho_3^{n+1}$. Next, \eqref{eq:rho_2} is discretized as follows
\begin{equation}
\begin{split}
\frac{\rho_2^{n+1} - \rho_2^{n}}{\tau} + \frac{\tau}{2} \frac{\rho_3^{n+1}- \rho_3^n}{\tau}  + \tilde{\bu}_0^{n+1} \cdot \nabla \rho_2^{n+1}  +  2 \tilde{\bu}_1^{n+1} \cdot \nabla \tilde{\rho}_1^{n+1}  + \tilde{\bu}_2^{n+1} \cdot \nabla \tilde{\rho}_0^{n+1} = \sigma_{h,2} \Delta \rho_2^{n+1},
\end{split}
\end{equation}
where, again, the term $ \sigma_{h,2} \Delta \rho_2^{n+1}$ is an artificial viscosity term and the correction term $ \frac{\tau}{2} \frac{\rho_3^{n+1} - \rho_3^n}{\tau}$ makes sure that $\rho_2^{n+1}$ is second-order accurate in time. Next, \eqref{eq:rho_1} is discretized leading to
\begin{equation}
\begin{split}
 \frac{\rho_1^{n+1}-\rho_1^{n}}{\tau}  + \frac{\tau}{2} \frac{\rho_2^{n+1} - \rho_2^n}{\tau} + \frac{\tau^2}{12} \frac{\rho_3^{n+1} - \rho_3^n}{\tau}  + \tilde{\bu}_0^{n+1} \cdot \nabla \rho_1^{n+1}  + \tilde{\bu}_1^{n+1} \cdot \nabla \tilde{\rho}_0^{n+1} =  \sigma_{h,1} \Delta \rho_1^{n+1}, \label{eq:1_con_update}
\end{split}
\end{equation}
where, again, $\sigma_{h,1} \Delta \rho_1^{n+1}$ is an artificial viscosity term and the correction terms $ \frac{\tau}{2} \frac{\rho_2^{n+1} + \rho_2^n}{\tau} - \frac{\tau^2}{12} \frac{\rho_3^{n+1} - \rho_3^n}{\tau}$ ensure that $\rho_1^{n+1}$ is third-order accurate in time. Lastly, \eqref{eq:rho_last} is discretized leading to
\begin{equation}
\begin{split}
\frac{\rho_0^{n+1}-\rho_0^{n}}{\tau} + \frac{\tau}{2} \frac{\rho_1^{n+1} - \rho_1^n}{\tau} + \frac{\tau^2}{12} \frac{\rho_2^{n+1} - \rho_2^n}{\tau} + \tilde{\bu}_0^{n+1} \cdot \nabla \rho_0^{n+1} =  \sigma_{h} \Delta \rho_0^{n+1}  , \label{eq:final_con_update}
\end{split}
\end{equation}
where, again, $ \sigma_{h} \Delta \rho_0^{n+1}$ is an artificial viscosity term and the correction terms $\frac{\tau}{2} \frac{\rho_1^{n+1} - \rho_1^n}{\tau} + \frac{\tau^2}{12} \frac{\rho_2^{n+1} - \rho_2^n}{\tau}$ ensure that $\rho_0^{n+1}$ is fourth-order accurate in time. How the artificial viscosities ${\sigma}_{h,3}$, ${\sigma}_{h,2}$, ${\sigma}_{h,1}$ and ${\sigma}_{h}$ are chosen is explained in the numerical experiments in Section \ref{Sec:computations}.

\subsubsection{Velocity and pressure update} \label{Sec:high_order_mom_update}
In this section we present the time-stepping algorithm for the momentum update. The aim is to discretize \eqref{eq:mom_differentiation_3}-\eqref{eq:mom_differentiation_0} using the Taylor series technique. We follow the exact same technique as \citep{Guermond_Minev_2019}, i.e. the divergence-free constraint is imposed in the same way and we treat the nonlinear terms explicitly with appropriate finite differences. We propose the following discretization of \eqref{eq:mom_differentiation_3}: Given $\bu_0^n, \bu_1^n, \bu_2^n, \bu_3^n$ and $\rho_0^{n+1}, \rho_1^{n+1}, \rho_2^{n+1}, \rho_3^{n+1}$, compute $\bu_3^{n+1}$ and $p_3^{n+1}$ as follows:
 \begin{equation} \label{eq:3_mom_update}
 \begin{split}
\rho_0^{n+1} \l( \frac{\bu_3^{n+1} - \bu_3^{n}}{\tau} +  \frac{N(\tilde{\bu}_0^{n+3}) - 3 N(\tilde{\bu}_0^{n+2} )+ 3 N(\tilde{\bu}_0^{n+1} ) - N(\bu_0^n)}{\tau^3} \r) \\ + 3 \rho_1^{n+1} \l(  \bu_3^{n+1} + \frac{N(\tilde{\bu}_0^{n+2}) - 2 N(\tilde{\bu}_0^{n+1} ) + N(\bu_0^n)}{\tau^2} \r) \\ + 3 \rho_2^{n+1} \l( \tilde{\bu}_2^{n+1} + \frac{ -\frac{1}{6}N( \tilde{\bu}_0^{n+3}) + N( \tilde{\bu}_0^{n+2} ) - \frac{1}{2} N( \tilde{\bu}_0^{n+1} ) - \frac{1}{3} N(\bu_0^n ) }{\tau } \r) \\ + \rho_3^{n+1} \l( \tilde{\bu}_1^{n+1} +  N \left(\tilde{\bu}_0^{n+1} \right) \r)  - {\nu}_{h,3} \rho_0^{n+1} \Delta \bu_3^{n+1}  - \lambda \nabla \nabla \cdot \bu_3^{n+1} + \nabla p_3^n   = \bef_3^{n+1},
\end{split}
\end{equation}
\begin{align}
 \begin{split}
p_3^{n+1} = p_3^n - \lambda \nabla \cdot \bu^{n+1}_3,
\end{split}
\end{align}
where the term $ - {\nu}_{h,3} \rho_0^{n+1}  \Delta \bu_3^{n+1}$ is an artificial viscosity term used to stabilize $\bu_3^{n+1}$ in space. Next, \eqref{eq:mom_differentiation_2} is discretized as follows
\begin{equation} \label{eq:2_mom_update}
\begin{split}
 \rho_0^{n+1} \l( \frac{\bu_2^{n+1} - \bu_2^{n}}{\tau} +  \frac{\tau}{2} \frac{\bu_3^{n+1} - \bu_3^{n}}{\tau} + \frac{N(\tilde{\bu}_0^{n+2}) - 2 N(\tilde{\bu}_0^{n+1} ) + N(\bu_0^n)}{\tau^2}\r) \\ + 2 \rho_1^{n+1} \l( \bu_2^{n+1} + \frac{ -\frac{1}{6}N( \tilde{\bu}_0^{n+3}) + N( \tilde{\bu}_0^{n+2} ) - \frac{1}{2} N( \tilde{\bu}_0^{n+1} ) - \frac{1}{3} N(\bu_0^n ) }{\tau }  \r) \\
 + \rho_2^{n+1} \l( \tilde{\bu}_1^{n+1} + N \left(\tilde{\bu}_0^{n+1} \right) \r)  - {\nu}_{h,2} \rho_0^{n+1}  \Delta \bu_2^{n+1}  - \lambda \nabla \nabla \cdot \bu_2^{n+1} +  \nabla (p_2^n + \tau p_3^{n+1} )  =   \bef_2^{n+1},
\end{split}
\end{equation}
\begin{align}
\begin{split}
p_2^{n+1} = p_2^n + \tau p_3^{n+1} - \lambda \nabla \cdot \bu^{n+1}_2,
\end{split}   
\end{align}
where, again, $ {\nu}_{h,2} \rho_0^{n+1}  \Delta \bu_2^{n+1}$ is an artificial viscosity term and the correction term $\frac{\tau}{2} \frac{\bu_3^{n+1} - \bu_3^{n}}{\tau}$ ensures that $\bu_2^{n+1}$ is second-order accurate in time. Next, \eqref{eq:mom_differentiation_1} is discretized as follows
\begin{equation}
\begin{split}
 & \rho_0^{n+1} \l( \frac{\bu_1^{n+1}-\bu_1^{n}}{\tau} + \frac{\tau}{2} \frac{\bu_2^{n+1} -\bu_2^n}{\tau} + \frac{\tau^2}{12} \frac{\bu_3^{n+1} -\bu_3^n}{\tau}  \right. \\  & \l.
 +\frac{ -\frac{1}{6}N( \tilde{\bu}_0^{n+3})  + N( \tilde{\bu}_0^{n+2} ) - \frac{1}{2} N( \tilde{\bu}_0^{n+1} ) - \frac{1}{3} N(\bu_0^n ) }{\tau } \r) \\ &
 + \rho_1^{n+1} \l( \bu_1^{n+1} +  N \left(\tilde{\bu}_0^{n+1} \right) \r) - {\nu}_{h,1} \rho_0^{n+1}  \Delta \bu_1^{n+1}  - \lambda \nabla \nabla \cdot \bu_1^{n+1} + \nabla \l( p_1^n + \tau p_2^{n+1} - \frac{\tau^2}{2}p_3^{n+1}  \r)  =  \bef_1^{n+1}, \label{eq:1_mom_update}  
\end{split}
\end{equation}
\begin{align}
\begin{split}
p_1^{n+1} = p_1^n + \tau p_2^{n+1} - \frac{\tau^2}{2} p_3^{n+1}   - \lambda \nabla \cdot \bu^{n+1}_1,
\end{split}
\end{align}
where, again, $ {\nu}_{h,1} \rho_0^{n+1}  \Delta \bu_1^{n+1}$ is an artificial viscosity term and the correction terms $\\ \frac{\tau}{2} \frac{\bu_2^{n+1} -\bu_2^n}{\tau} + \frac{\tau^2}{12} \frac{\bu_3^{n+1} -\bu_3^n}{\tau}$ ensure that $\bu_1^{n+1}$ is third-order accurate in time. Lastly, \eqref{eq:mom_differentiation_0} is discretized as follows
\begin{equation}
\begin{split}
 \rho_0^{n+1} \l( \frac{\bu_0^{n+1}-\bu_0^{n}}{\tau}  +  \frac{\tau}{2} \frac{\bu_1^{n+1} - \bu_1^{n}}{\tau} + \frac{\tau^2}{12} \frac{\bu_2^{n+1} - \bu_2^n}{\tau} +  N \left(\tilde{\bu}_0^{n+1} \right) \r)\\ - {\nu}_{h} \rho_0^{n+1}  \Delta \bu_0^{n+1} - \lambda \nabla \nabla \cdot \bu_0^{n+1} + \nabla \left( p_0^n + \tau p_1^{n+1} - \frac{\tau^2}{2} p_2^{n+1} + \frac{\tau^3}{6} p_3^{n+1} \right) = \bef_0^{n+1} , \label{eq:final_mom_update} 
\end{split}
\end{equation}
\begin{align}
\begin{split}
 p_0^{n+1} = p_0^n + \tau p_1^{n+1} - \frac{\tau^2}{2} p_2^{n+1} + \frac{\tau^3}{6} p_3^{n+1} - \lambda \nabla \cdot \bu^{n+1}_0,
\end{split}
\end{align}
where, again, ${\nu}_{h} \rho_0^{n+1}  \Delta \bu_0^{n+1}$ is an artificial viscosity term and the correction terms $ \frac{\tau}{2} \frac{\bu_1^{n+1} - \bu_1^{n}}{\tau} + \frac{\tau^2}{12} \frac{\bu_2^{n+1} - \bu_2^n}{\tau} $ ensures that $\bu_0^{n+1}$ is fourth-order accurate in time. How the artificial viscosities \commentB{${\nu}_{h,3}$, ${\nu}_{h,2}$, ${\nu}_{h,1}$ and ${\nu}_{h}$} are chosen is explained in the numerical experiments in Section \ref{Sec:computations}.

\subsection{Initialization} \label{Sec:initialitaztion}
Ideally, if initial conditions for the solutions and their time derivatives are available they should be used. This is not generally the case though, and to this end, \cite{Guermond_Minev_2019} provided an initialization technique based on Richardson extrapolation. In this section, we extend their technique to fourth-order accuracy and variable density. 

First, for a given initial velocity $\bu^0(\bx)$ and density $\rho^0 ( \bx )$, the initial pressure $p^0(\bx)$ is obtained by solving 

\begin{equation} \label{eq:init_pressure_varden}
\begin{split}
\nabla \cdot \l( \l( \rho^0 \r)^{-1} \nabla p^0 \r) = \nabla \cdot \l( \l( \rho^0 \r)^{-1} \l( \bef(0) + \mu \Delta \bu^0 \r) - \bu^0 \cdot \nabla \bu^0 \r) \\  \partial_n p^0 = (\bef(0) + \mu \Delta \bu^0 - \rho^0 \bu^0 \cdot \nabla \bu^0) \cdot \bn |\Gamma.
\end{split}
\end{equation}

The next step is to use the first-order algorithm from Section \ref{Sec:fem_approx} to construct solutions from different time levels using different time-steps. That different time-steps are used is important to be able to use Richardson extrapolation. This is used to construct fourth-order accurate solutions. Lastly, using appropriately accurate finite difference formulas, the time derivatives are constructed. For additional details on the motivation behind this procedure, we refer to \citep[Sec 3.3]{Guermond_Minev_2019}.

Let $(\rho_{\tau/k}^n,\bu_{\tau/k}^n, p_{\tau/k}^n)$ be solutions obtained at $t = n \tau$ with the time-step $\tau/k$ using the first-order algorithm from Section \ref{Sec:fem_approx}. The initialization algorithm proceeds as follows:

\begin{algorithm}[H] 
\caption{Initialization algorithm.} \label{rich_init}
\tcc{Initial pressure}

Compute $p^0$ by solving \eqref{eq:init_pressure_varden}.

\tcc{Use first-order algorithm from Section \ref{Sec:fem_approx} to compute solutions using time-steps $\tau/4$, $\tau/3$, $\tau/2$ and $\tau$.}

Compute $(\rho_{\tau/4}^1,\bu_{\tau/4}^1, p_{\tau/4}^1)$, $(\rho_{\tau/4}^2,\bu_{\tau/4}^2, p_{\tau/4}^2)$ and $(\rho_{\tau/4}^3,\bu_{\tau/4}^3, p_{\tau/4}^3)$.

Compute $(\rho_{\tau/3}^1,\bu_{\tau/3}^1, p_{\tau/3}^1)$, $(\rho_{\tau/3}^2,\bu_{\tau/3}^2, p_{\tau/3}^2)$ and $(\rho_{\tau/3}^3,\bu_{\tau/3}^3, p_{\tau/3}^3)$.

Compute $(\rho_{\tau/2}^1,\bu_{\tau/2}^1, p_{\tau/2}^1)$, $(\rho_{\tau/2}^2,\bu_{\tau/2}^2, p_{\tau/2}^2)$ and $(\rho_{\tau/2}^3,\bu_{\tau/2}^3, p_{\tau/2}^3)$.

Compute $(\rho_{\tau}^1,\bu_{\tau}^1, p_{\tau}^1)$, $(\rho_{\tau}^2,\bu_{\tau}^2, p_{\tau}^2)$ and $(\rho_{\tau}^3,\bu_{\tau}^3, p_{\tau}^3)$.

\tcc{Compute fourth-order accurate solutions $(\rho_{R}^{1,2,3},\bu_{R}^{1,2,3}, p_{R}^{1,2,3})$ using a Richardson extrapolation formula, where the subscript $R$ denotes that the solution has been computed using the Richardson extrapolation formula below.}

$\phi_R^n = \frac{5}{42} \phi_{\tau}^n + \frac{4}{7} \phi_{\tau/2}^n - \frac{81}{14} \phi_{\tau/3}^n + \frac{128}{21} \phi_{\tau/4}^n$

\tcc{Compute initial conditions of the time derivatives using finite difference formulas.}


$\phi_t(\tau)= \frac{- \phi_R^3 + 6 \phi_R^2 - 3 \phi_R^1 -2 \phi(0)}{6 \tau} $

$\phi_{tt}(\tau)= \frac{ \phi_R^2 - 2 \phi_R^1 + \phi(0)}{ \tau^2} $

$\phi_{ttt}(\tau)= \frac{ \phi_R^3 - 3 \phi_R^2 + 3 \phi_R^1 - \phi(0)}{\tau^3} $
%
%
%
%

\end{algorithm}

\subsection{Time adaptivity} \label{Sec:adaptivity} 
There are many approaches to time adaptivity. In this manuscript, we don’t focus on this and only follow the simple time-step control algorithm proposed by \cite[Sec 5.4]{Guermond_Minev_2019}. Given a user-specified tolerance TOL, Algorithm \ref{algorithm_variable_density} describes how to choose each time-step based on the CFL condition
and an estimation of the local error. In addition to the CFL number and the tolerance TOL, the user has to specify two constants: the maximum growth rate $s_{max}$ and
minimum decrease rate $s_{min}$ of each calculated time-step. If the next calculated time
step has decreased by less than $s_{min}$, the computation should be redone. Since our focus is not on time-adaptivity, we only choose timestep based on stability, i.e. CFL condition, and not based on error tolerances. In practice this means
that we choose $\textrm{TOL} = \infty$  in all our numerical examples.

\begin{algorithm}[H]
\caption{Time step control algorithm \citep[Sec 5.4]{Guermond_Minev_2019} extended to variable density.} \label{algorithm_variable_density}
\tcc{Compute time-step based on CFL condition}
$s_{cfl} = \textrm{CFL} \min_{\bx \in \Omega} h(\bx) / (\| \bu_0^n (\bx)\|_{l^2} \tau^n)$

\tcc{Estimate local time error for density}
$e_{loc, \rho} = ...$

\tcc{Estimate local time error for velocity}
$e_{loc, \bu} = ...$

\tcc{Compute time-step increment based on user-specified tolerance}

$s_{\bu} = \textrm{TOL}/\l( e_{loc, \bu} \tau^n \r)$ 

$s_{\rho} = \textrm{TOL}/\l( e_{loc, \rho} \tau^n \r)$ 

\tcc{Set next time-step based on both stability and user-specified tolerance}
$s = \min{ \l( s_{cfl}, s_{\bu}, s_{\rho}, s_{max} \r) }$

$\tau^{n+1} = s \tau^n$

\If{$s < s_{min}$}{Repeat previous time-step with $\tau^{n+1}$ instead}

\Return{Time step $\tau^{n+1}$ and flag whether to repeat time-step or not}
\end{algorithm}

%
\section{Boundary conditions} \label{Sec:boundary_conditions}
The boundary conditions in this work are only imposed on velocity and density. For velocity, the problems we solve only involve Dirichlet boundary conditions and slip boundary conditions. In both cases, these boundary conditions are imposed strongly through the linear system. For details on how to impose slip boundary conditions in this way, we refer to \cite[Sec 4.1.2]{Nazarov_Larcher_2017}.

The velocity field in the problems we consider in this work fulfills the criteria $\bu \cdot \bn = 0$, therefore no Dirichlet boundary condition for the density should be set since there is no inflow.

\section{Numerical examples} \label{Sec:computations}
In this section, we test the method against some benchmarks from the literature. The temporal discretization is described in Sections \ref{Sec:high_order_cont_update} and \ref{Sec:high_order_mom_update}. The spatial discretization for each sub-stage of the time-stepping algorithm is described in Section \ref{Sec:fem_approx}. Since the time-stepping algorithm is fourth-order accurate we use $\polP_3$ finite elements for the spatial discretization to get an expected accuracy of four in space and time. To satisfy the inf-sup condition we use $\polP_2$ elements for the pressure.

\subsection{Accuracy test} \label{Sec:accuracy test}
In this section we verify the accuracy of the proposed method by using a manufactured solution on a unit disk. We follow the setup from \cite{Guermond_Salgado_2009} where the forcing function $\bef$ is chosen to obtain the following exact solution
\begin{equation}
\begin{split}
&\rho(\bx,t) = 2 + x \cos( \sin(t) ) + y \sin ( \sin(t) ),\\
&\bu(\bx,t) =  \begin{bmatrix}
-y \cos(t) \\
x \cos(t)
\end{bmatrix}, \\
&p(\bx,t) = \sin(x) \sin(y) \sin(t).
\end{split}
\end{equation}

We perform a convergence study using a series of unstructured meshes. The dynamic viscosity is set to $\mu = 1$ and we set the stabilization coefficients used in Sections \ref{Sec:high_order_cont_update} and \ref{Sec:high_order_mom_update} to zero (i.e. $\sigma_{h,1,2,3} = 0$ and $\nu_{h,1,2,3} = 0$). The termination time is set to $T = 10$. The time-step is set to be constant $\Delta t = \sqrt{2} h_{min}$ which corresponds to CFL = 1, since $\max_{(\Omega \times [0,T])} \| \bu \|_{l^2} = \sqrt{2}$. The minimum grid spacing $h_{min}$ is given by $h_{min} := \min_{K \in \mathcal{T}_h} h_K$. 

Due to the nature of the time-stepping algorithm, the corresponding analytic time derivatives of $\bef$ were used for each time derivative update. Additionally, we set Dirichlet boundary conditions for the velocity and its time derivatives. We present convergence results in Tables \ref{table:velocity_P3P3P2_ic}-\ref{table:pressure_P3P3P2}. In Tables \ref{table:velocity_P3P3P2_ic}, \ref{table:density_P3P3P2_ic} and \ref{table:pressure_P3P3P2_ic} the analytic time derivatives were used as initial conditions for the time derivatives. For the purpose of validation, we also present results where the time derivatives were initialized using the Richardson initialization algorithm (Algorithm \ref{rich_init}) in Tables \ref{table:velocity_P3P3P2}, \ref{table:density_P3P3P2} and \ref{table:pressure_P3P3P2}. 

The $L_1$, $L_2$ and $L_\infty$ errors are presented for all components. All the errors are computed using high-order quadrature and are relative, i.e. they are normalized with their corresponding norm. If the time derivatives are used as an initial condition we were able to obtain the expected high-order convergence rate using $\lambda = 1$ similar to \cite[Sec 5.1]{Guermond_Minev_2019}. However, when Algorithm \ref{rich_init} was used as initialization, we were unable to achieve this using a small $\lambda$. By increasing $\lambda$ it was possible to obtain the correct convergence rate. In Section \ref{Sec:lambda_study} we investigate how $\lambda$ affects the error in more detail by doing a parameter sweep.


\begin{table}[H]
\centering     
\caption{Velocity $\polP_3 \polP_3 \polP_2$. $\lambda = 1$. Initial condition used for all time derivatives.}
\begin{tabular}{|c|c|c|c|c|c|c|c|}    
\hline
&\multicolumn{1}{c|}{$\#$ dofs} & \multicolumn{1}{c|}{$L_1$} & \multicolumn{1}{c|}{rate} & \multicolumn{1}{c|}{$L_2$} & \multicolumn{1}{c|}{rate} & \multicolumn{1}{c|}{$L_{\infty}$} & \multicolumn{1}{c|}{rate}\\
 \hline
\parbox[t]{2mm}{\multirow{6}{*}{\rotatebox[origin=c]{90}{Galerkin}}} & 
       200 &   1.84E-03 &   0.00 &   2.11E-03 &   0.00 &   2.43E-03 &   0.00 \\ 
&       386 &   7.41E-04 &   2.76 &   7.39E-04 &   3.19 &   6.68E-04 &   3.93 \\ 
&       794 &   1.83E-04 &   3.87 &   1.89E-04 &   3.79 &   2.06E-04 &   3.26 \\ 
&      3314 &   1.17E-05 &   3.85 &   1.36E-05 &   3.68 &   1.56E-05 &   3.61 \\ 
&     13184 &   7.03E-07 &   4.07 &   6.96E-07 &   4.31 &   6.46E-07 &   4.61 \\ 
&     52184 &   2.34E-08 &   4.95 &   2.89E-08 &   4.62 &   3.52E-08 &   4.23 \\  \hline
\end{tabular}                                                                   
\label{table:velocity_P3P3P2_ic}                                                      
\end{table}

\begin{table}[H]
\centering     
\caption{Density $\polP_3 \polP_3 \polP_2$. $\lambda = 1$. Initial condition used for all time derivatives.}
\begin{tabular}{|c|c|c|c|c|c|c|c|}    
\hline
&\multicolumn{1}{c|}{$\#$ dofs} & \multicolumn{1}{c|}{$L_1$} & \multicolumn{1}{c|}{rate} & \multicolumn{1}{c|}{$L_2$} & \multicolumn{1}{c|}{rate} & \multicolumn{1}{c|}{$L_{\infty}$} & \multicolumn{1}{c|}{rate}\\
 \hline
\parbox[t]{2mm}{\multirow{6}{*}{\rotatebox[origin=c]{90}{Galerkin}}} & 
       100 &   2.37E-03 &   0.00 &   2.69E-03 &   0.00 &   3.32E-03 &   0.00 \\ 
&       193 &   7.47E-04 &   3.51 &   8.47E-04 &   3.51 &   1.06E-03 &   3.48 \\ 
&       397 &   1.31E-04 &   4.83 &   1.49E-04 &   4.81 &   1.89E-04 &   4.77 \\ 
&      1657 &   3.03E-06 &   5.27 &   3.68E-06 &   5.19 &   5.40E-06 &   4.98 \\ 
&      6592 &   1.64E-07 &   4.23 &   1.91E-07 &   4.28 &   2.59E-07 &   4.40 \\ 
&     26092 &   1.18E-08 &   3.83 &   1.37E-08 &   3.83 &   1.85E-08 &   3.84 \\   \hline
\end{tabular}                                                                   
\label{table:density_P3P3P2_ic}                                                      
\end{table}

\begin{table}[H]
\centering     
\caption{Pressure $\polP_3 \polP_3 \polP_2$. $\lambda = 1$. Initial condition used for all time derivatives.}
\begin{tabular}{|c|c|c|c|c|c|c|c|}    
\hline
&\multicolumn{1}{c|}{$\#$ dofs} & \multicolumn{1}{c|}{$L_1$} & \multicolumn{1}{c|}{rate} & \multicolumn{1}{c|}{$L_2$} & \multicolumn{1}{c|}{rate} & \multicolumn{1}{c|}{$L_{\infty}$} & \multicolumn{1}{c|}{rate}\\
 \hline
\parbox[t]{2mm}{\multirow{6}{*}{\rotatebox[origin=c]{90}{Galerkin}}} & 
        48 &   7.98E-02 &   0.00 &   7.79E-02 &   0.00 &   1.71E-01 &   0.00 \\ 
&        91 &   2.68E-02 &   3.41 &   2.48E-02 &   3.58 &   2.38E-02 &   6.16 \\ 
&       184 &   6.81E-03 &   3.89 &   6.26E-03 &   3.92 &   5.63E-03 &   4.10 \\ 
&       751 &   7.53E-04 &   3.13 &   6.96E-04 &   3.12 &   7.08E-04 &   2.95 \\ 
&      2958 &   4.06E-05 &   4.26 &   3.81E-05 &   4.24 &   5.43E-05 &   3.75 \\ 
&     11652 &   1.59E-06 &   4.73 &   1.74E-06 &   4.50 &   3.79E-06 &   3.88 \\  \hline
\end{tabular}                                                                   
\label{table:pressure_P3P3P2_ic}                                                      
\end{table}

\begin{table}[H]
\centering     
\caption{Velocity $\polP_3 \polP_3 \polP_2$. $\lambda = 10000$. Richardson initialization.}
\begin{tabular}{|c|c|c|c|c|c|c|c|}    
\hline
&\multicolumn{1}{c|}{$\#$ dofs} & \multicolumn{1}{c|}{$L_1$} & \multicolumn{1}{c|}{rate} & \multicolumn{1}{c|}{$L_2$} & \multicolumn{1}{c|}{rate} & \multicolumn{1}{c|}{$L_{\infty}$} & \multicolumn{1}{c|}{rate}\\
 \hline
\parbox[t]{2mm}{\multirow{6}{*}{\rotatebox[origin=c]{90}{Galerkin}}} & 
       200 &   2.21E-06 &   0.00 &   2.40E-06 &   0.00 &   2.78E-06 &   0.00 \\ 
&       386 &   3.49E-07 &   5.62 &   3.72E-07 &   5.67 &   4.62E-07 &   5.46 \\ 
&       794 &   7.53E-08 &   4.25 &   8.20E-08 &   4.19 &   1.05E-07 &   4.10 \\ 
&      3314 &   3.58E-09 &   4.26 &   3.85E-09 &   4.28 &   4.96E-09 &   4.28 \\ 
&     13184 &   2.23E-10 &   4.02 &   2.49E-10 &   3.97 &   3.43E-10 &   3.87 \\ 
&     52184 &   4.87E-10 &  -1.13 &   5.01E-10 &  -1.02 &   5.57E-10 &  -0.71 \\ \hline
\end{tabular}                                                                   
\label{table:velocity_P3P3P2}                                                      
\end{table}

\begin{table}[H]
\centering     
\caption{Density $\polP_3 \polP_3 \polP_2$. $\lambda = 10000$. Richardson initialization.}
\begin{tabular}{|c|c|c|c|c|c|c|c|}    
\hline
&\multicolumn{1}{c|}{$\#$ dofs} & \multicolumn{1}{c|}{$L_1$} & \multicolumn{1}{c|}{rate} & \multicolumn{1}{c|}{$L_2$} & \multicolumn{1}{c|}{rate} & \multicolumn{1}{c|}{$L_{\infty}$} & \multicolumn{1}{c|}{rate}\\
 \hline
\parbox[t]{2mm}{\multirow{6}{*}{\rotatebox[origin=c]{90}{Galerkin}}} & 
       100 &   8.50E-05 &   0.00 &   9.74E-05 &   0.00 &   1.58E-04 &   0.00 \\ 
&       193 &   1.34E-05 &   5.61 &   1.53E-05 &   5.64 &   2.06E-05 &   6.20 \\ 
&       397 &   2.30E-06 &   4.90 &   2.59E-06 &   4.92 &   3.35E-06 &   5.04 \\ 
&      1657 &   1.19E-07 &   4.15 &   1.35E-07 &   4.14 &   1.83E-07 &   4.07 \\ 
&      6592 &   5.27E-09 &   4.51 &   6.23E-09 &   4.45 &   9.48E-09 &   4.28 \\ 
&     26092 &   4.63E-10 &   3.54 &   5.72E-10 &   3.47 &   9.85E-10 &   3.29 \\   \hline
\end{tabular}                                                                   
\label{table:density_P3P3P2}                                                 
\end{table}

\begin{table}[H]
\centering     
\caption{Pressure $\polP_3 \polP_3 \polP_2$. $\lambda = 10000$. Richardson initialization.}
\begin{tabular}{|c|c|c|c|c|c|c|c|}    
\hline
&\multicolumn{1}{c|}{$\#$ dofs} & \multicolumn{1}{c|}{$L_1$} & \multicolumn{1}{c|}{rate} & \multicolumn{1}{c|}{$L_2$} & \multicolumn{1}{c|}{rate} & \multicolumn{1}{c|}{$L_{\infty}$} & \multicolumn{1}{c|}{rate}\\
 \hline
\parbox[t]{2mm}{\multirow{6}{*}{\rotatebox[origin=c]{90}{Galerkin}}} & 
        48 &   3.18E-03 &   0.00 &   4.89E-03 &   0.00 &   1.36E-02 &   0.00 \\ 
&        91 &   1.38E-03 &   2.61 &   1.95E-03 &   2.87 &   5.09E-03 &   3.07 \\ 
&       184 &   4.49E-04 &   3.19 &   6.41E-04 &   3.16 &   1.55E-03 &   3.37 \\ 
&       751 &   5.51E-05 &   2.98 &   7.48E-05 &   3.05 &   2.18E-04 &   2.79 \\ 
&      2958 &   7.53E-06 &   2.90 &   9.94E-06 &   2.95 &   2.81E-05 &   2.99 \\ 
&     11652 &   9.54E-07 &   3.01 &   1.25E-06 &   3.02 &   3.03E-06 &   3.25 \\   \hline
\end{tabular}                                                                   
\label{table:pressure_P3P3P2}                                                      
\end{table}

\subsection{How $\lambda$ effects different errors} \label{Sec:lambda_study}
In this section, we investigate how $\lambda$ affects different errors of the solution. We use the same setup as in Section \ref{Sec:accuracy test} with an unstructured mesh consisting of 1657 $\polP_3$ nodes. In Figure \ref{fig:lambda_error} the maximum divergence error over space and time, i.e. $\|\nabla \cdot \bu_h\|_{L_\infty(\Omega \times (0,T))}$, is plotted against $\lambda$. Moreover, the relative $L_2$ error of the velocity and density at time $T= 10$ is plotted against $\lambda$. Overall, the figure indicates that increasing $\lambda$ decreases the divergence error and also decreases the error in density and velocity up to some threshold. Another important observation is that the Richardson initialization algorithm (Algorithm \ref{rich_init}) requires a larger $\lambda$ than when initial conditions for the time derivatives are used.


\begin{figure}[H]
\begin{subfigure}{.5\linewidth}
\centering 
\includegraphics[scale=0.5]{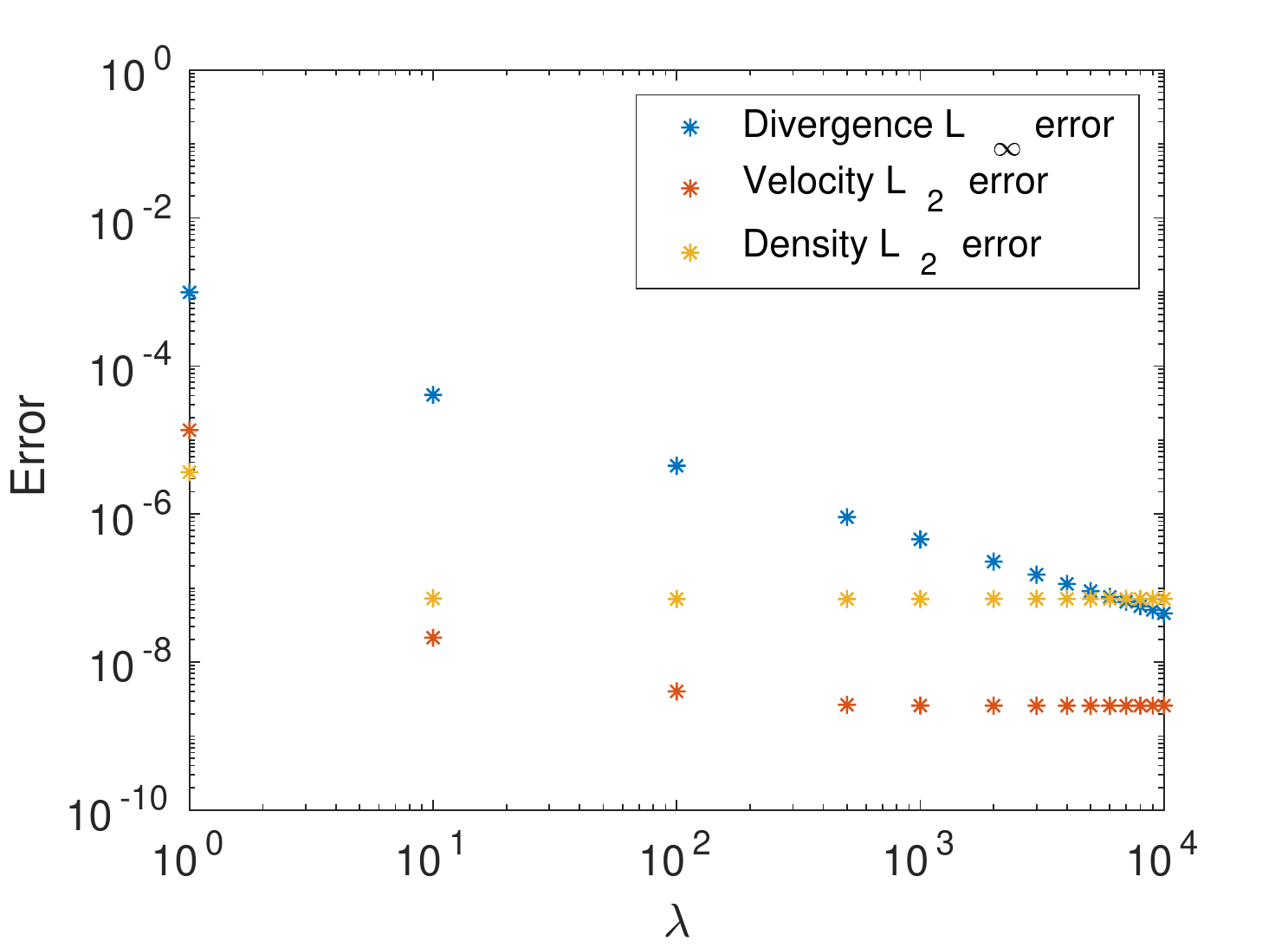}
\caption{}
\end{subfigure}%
\begin{subfigure}{.5\linewidth} 
\centering
\includegraphics[scale=0.5]{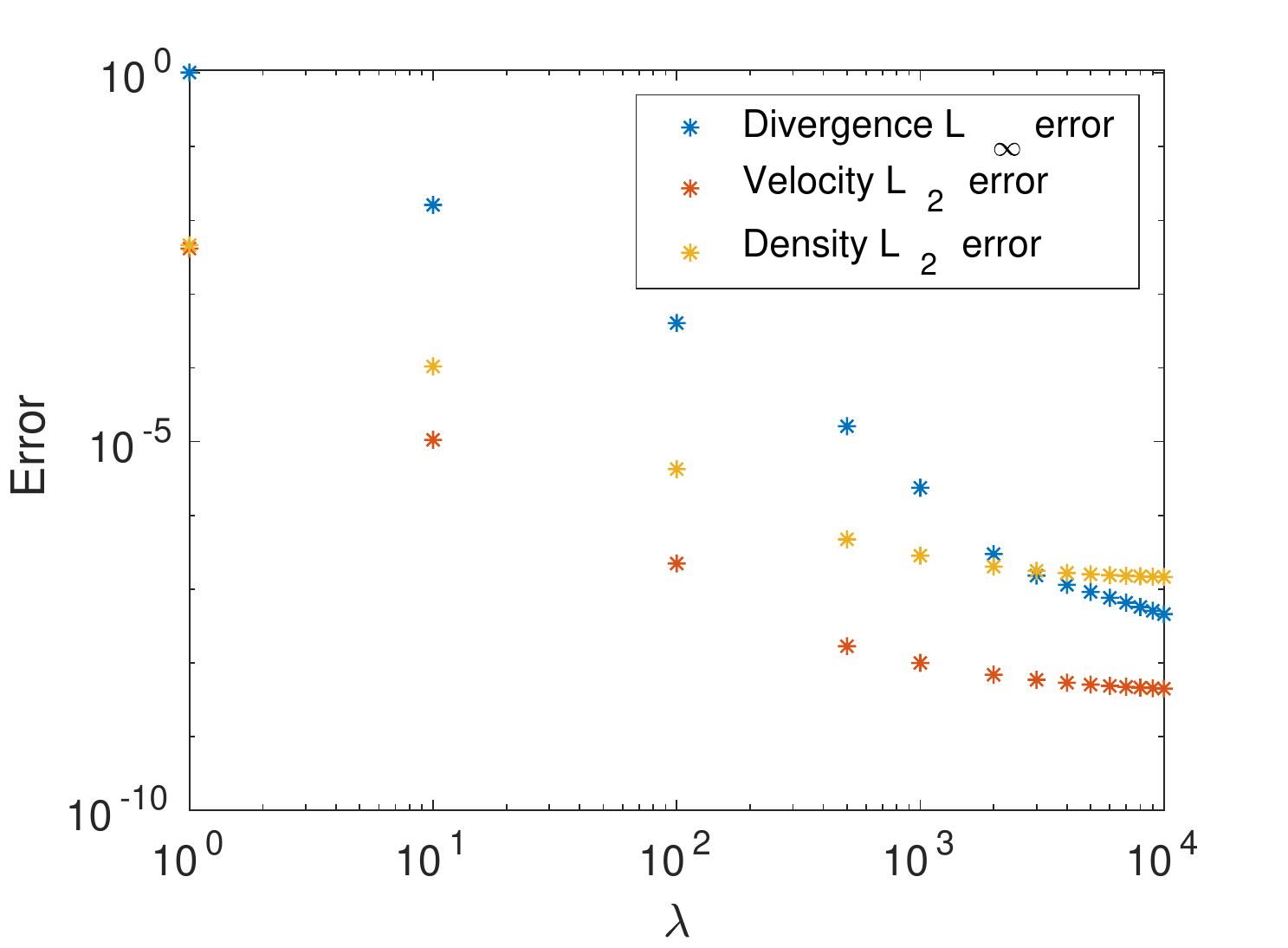}
\caption{}
\end{subfigure}
\caption{Error as a function of $\lambda$ using (a) analytic initial conditions and (b) Richardson initialization.}
\label{fig:lambda_error}
\end{figure}

\subsection{Loss of regularity} \label{Sec:rt_ins} 
In this section, we investigate how the Taylor series time-stepping method handles problems with sharp gradients, i.e. when regularity of the solution is lost. \commentA{The Taylor series method relies on computing time derivative correction terms to achieve high-order accuracy. Needing to compute time derivatives when the solution contains a discontinuity is problematic since the analytic time derivatives are infinitely large. If the discontinuity is regularized using a stabilization technique, computing the time derivatives might still be acceptable. We investigate this in Section \ref{sec:rayleigh} by considering a benchmark problem where regularity is lost.}

\commentA{We emphasize that there exist many stabilization methods which can be applied to the first-order method from Section \ref{Sec:fem_approx}. Some options include entropy viscosity \citep{Guermond_Pasquetti_2011,Nazarov_Larcher_2017,Guermond_Nazarov_2011}, residual viscosity \citep{Nazarov_Hoffman_2013,Stiernstrom2021,Lu2019,Marras2015,Dao2022}, first-order viscosity, Galerkin least-squares \citep{Hughes_1989} and so on. Stabilizing the time derivative updates from Section \ref{Sec:time-stepping-method} is an entirely different matter since the analytic time derivatives are unbounded in the presence of a discontinuity.}

\commentA{We choose to stabilize the method and its time derivative updates using $h$-viscosity, which is one of the most robust spatial stabilizations possible. Our aim is not to achieve an accurate result since the results will be overly diffusive and only first-order accurate in space. Instead, our aim is simply to try to stabilize the method and illustrate that even with this very diffuse stabilization the Taylor series method is not suited for problems of this kind. We impose $h$-viscosity by setting $\sigma_{h,1,2,3} = C_{max} h$ and $\nu_{h,1,2,3} = C_{max} h$ where $C_{max}$ is a constant, see Sections \ref{Sec:high_order_cont_update} and \ref{Sec:high_order_mom_update}.}


\begin{remark}
\commentA{For a hyperbolic problem in one space dimension using $\polP_1$ finite elements in space and explicit Euler in time, it can be shown that scaling the artificial viscosity coefficient as $\sigma_{h} = 0.5 | \bu_0^n | h$ leads to the classical first-order upwind scheme which is a convergent scheme \citep{Lax1954}}.  


\end{remark}

\begin{remark}
\commentA{A more suitable time-stepping method would be a method that allows for loss of regularity. Finding such a scheme in the context of high-order artificial compressibility is an open question. We suspect that the defect correction time-stepping approach from \citep[Sec 5.2]{Guermond_Minev_2015} suffers from a similar problem as the Taylor series method does. A viable alternative would be the modified high-order BDF scheme proposed in \citep[Sec 5.1]{Guermond_Minev_2015}. The downside would be that the scheme would lose its A-stability. One could also consider using a Lagrange-Galerkin approach. The Lagrange-Galerkin method discretizes the total derivative (the convective part of the equations) backward in time along the characteristic curves and has been successfully applied to hyperbolic problems \citep{Colera2021} and incompressible flow \citep{Bermejo2018,Bermejo2016}. Investigating if this time-stepping approach is amenable to high-order artificial compressibility would certainly be an interesting research topic.} 

\end{remark}



\subsubsection{2D Rayleigh-Taylor instability} \label{sec:rayleigh}
As a test problem where regularity is lost, we consider the Rayleigh-Taylor instability in 2D. The Rayleigh-Taylor instability occurs when a fluid accelerates into another fluid with a different density. The classical setup is that a heavier fluid is supported by a lighter fluid in a gravitational field. Any small perturbation to the system forces it out of equilibrium since the initial equilibrium state is unstable. We follow the same setup as \cite{Guermond_Salgado_2009}. The solution is computed in a rectangular domain $\Omega = \{ (x,y) \in (-d/2,d/2) \times (-2d,2d) \}$. The density jump is initially regularized using a hyperbolic tangent function and is given by
\begin{equation}
\rho(x,y,0) = \frac{\rho_1 + \rho_2}{2} + \frac{\rho_1 - \rho_2}{2} \tanh \l( \frac{y- \eta(\bx) }{0.01 d} \r),
\end{equation}
where $\eta(\bx) = - 0.1 d \cos(2\pi x /d)$. Slip boundary conditions are enforced at the boundaries. \commentB{The characteristic velocity scale is set to $\sqrt{dg}$ which gives the Reynolds number $Re = \rho_2 d^{\frac{3}{2}} g^{\frac{1}{2}} / \mu$}. The following parameters are used: $d = 1$, $\rho_1 = 3$, $\rho_2 = 1$, \commentB{$\bef = (0, -\rho g)$, $g = 1$,} $\lambda = 5000$, $s_{max} = 1.1$, $s_{min} = 0.75$, CFL = 0.5, \commentB{$\sigma_{h,1,2,3} = 0.5 h$, $\nu_{h,1,2,3} = 0.5 h$}. A sufficiently large value of $\lambda$ is set to ensure that divergence errors are negligible, see Section \ref{Sec:lambda_study} for more details. We use a structured mesh containing 45676 $\polP_3$ nodes, which later is uniformly refined two times.

In Figures \ref{fig:rt_n25}-\ref{fig:rt_n100} we present the time evolution of the computed density field for $Re = 5000$, respectively, in the time-scale of Tryggvason ($t = \sqrt{2} t_{Tryg}$). Overall, the results are in agreement with the result obtained by \cite{Guermond_Salgado_2009} where a second-order accurate Taylor-Hood finite element method was used. Since $h$-viscosity was used as stabilization the results presented here are very diffused. As a comparison we present results with the first-order algorithm from Section \ref{Sec:fem_approx} in Figure \ref{fig:rt_n100_order1}. Overall, the results of the first-order algorithm is completely identical to the Taylor series method.

In Figures \ref{fig:rt_rhotl_n25}-\ref{fig:rt_rhotl_n100} the computed third time derivative of the density is presented. We observe that the magnitude of the solution is increased when the grid is refined. The reason why the time derivative increases is that the gradient of the solution grows upon grid refinement. We infer that the time derivative will continue to grow without bound as the grid is refined. We, therefore, conclude that the Taylor series method is unsuitable for problems where the regularity of the solution is lost. \commentA{A more accurate and sharp stabilization procedure will exaggerate this issue.}

\newcommand \xWidth{.072}
\newcommand \xWidthcolorbar{\fpeval{1.4125*\xWidth}}
\newcommand \xWidthcolorbarr{\fpeval{0.75*\xWidth}}

\begin{figure}[H] \centering 
    \includegraphics[width=\xWidth \textwidth]{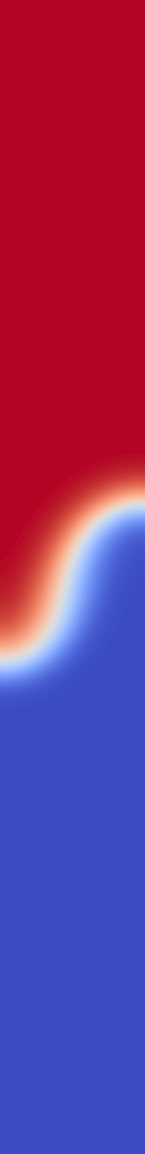}
    \includegraphics[width=\xWidth \textwidth]{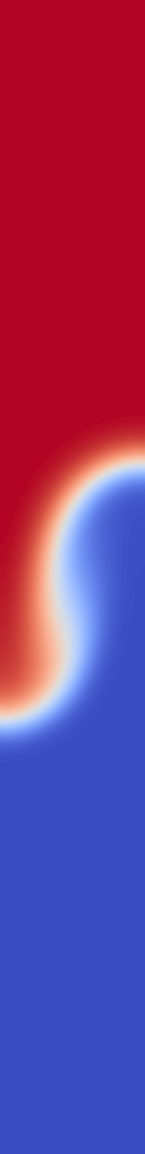}
    \includegraphics[width=\xWidth \textwidth]{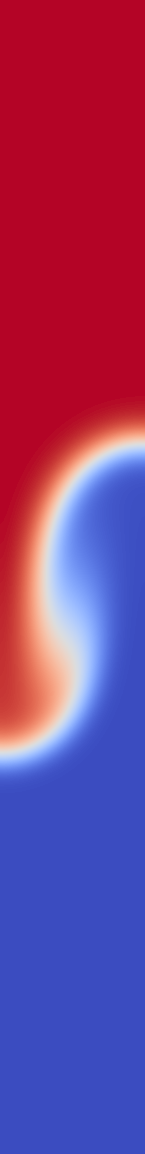}
    \includegraphics[width=\xWidth \textwidth]{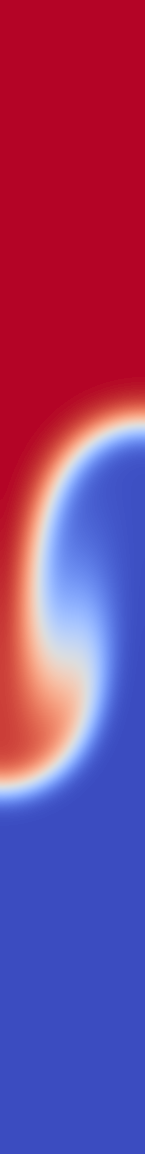}
    \includegraphics[width=\xWidth \textwidth]{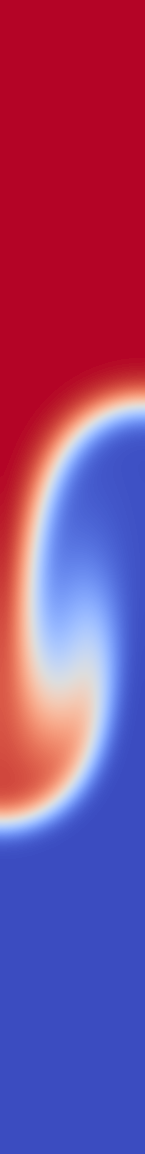}
    \includegraphics[width=\xWidth \textwidth]{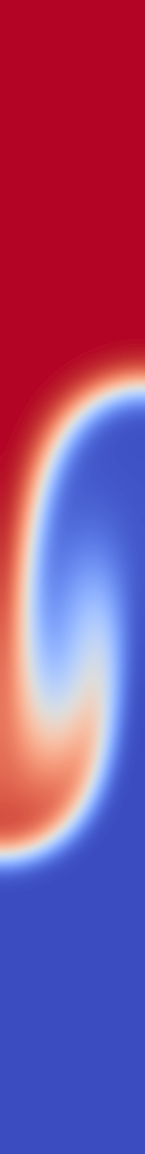}
    \includegraphics[width=\xWidthcolorbarr \textwidth]{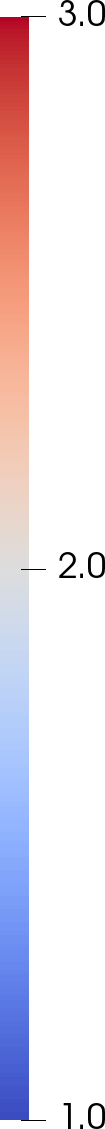}
  \caption{$Re = 5000$. Time evolution of $\rho$ using 45676 $\polP_3$ nodes at times 1, 1.5, 1.75, 2, 2.25 and 2.5.} \label{fig:rt_n25}
\end{figure}

\begin{figure}[H] \centering 
    \includegraphics[width=\xWidth \textwidth]{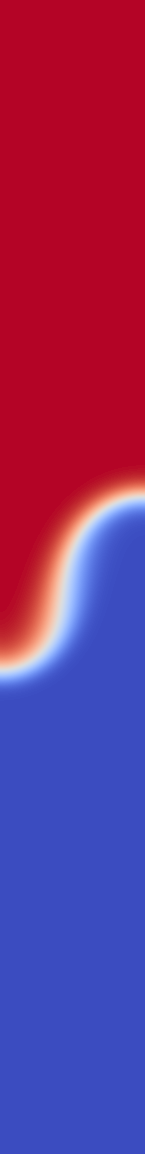}
    \includegraphics[width=\xWidth \textwidth]{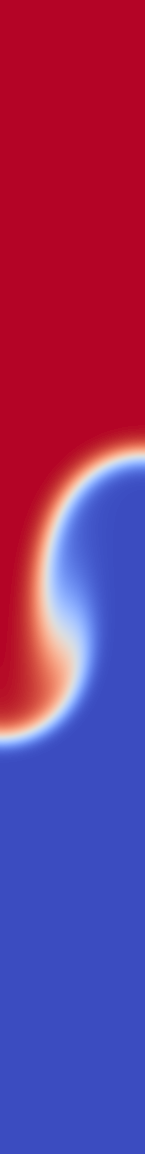}
    \includegraphics[width=\xWidth \textwidth]{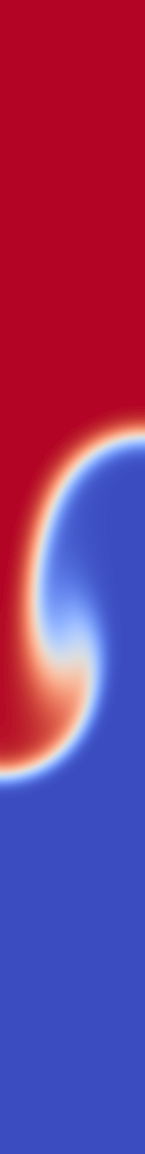}
    \includegraphics[width=\xWidth \textwidth]{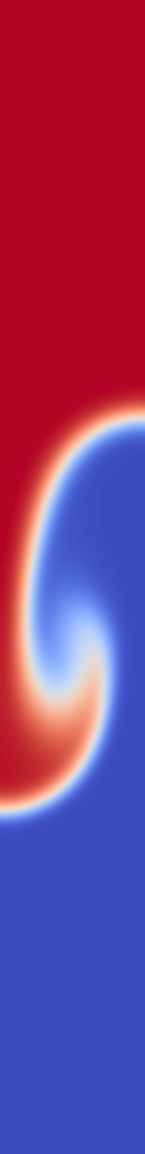}
    \includegraphics[width=\xWidth \textwidth]{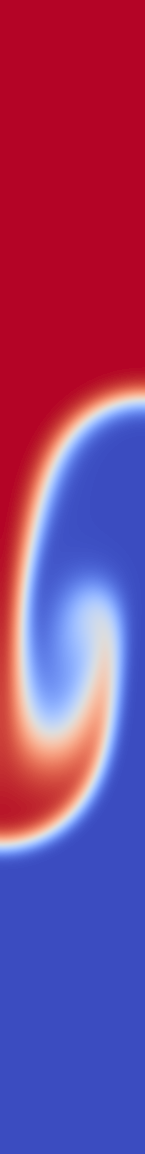}
    \includegraphics[width=\xWidth \textwidth]{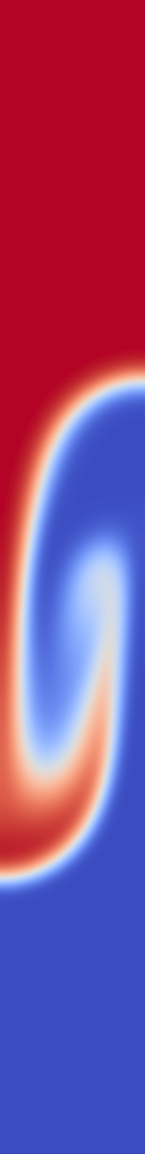}
    \includegraphics[width=\xWidthcolorbarr \textwidth]{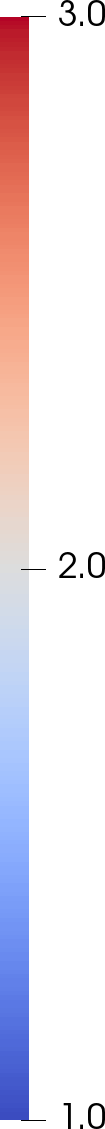}
  \caption{$Re = 5000$. Time evolution of $\rho$ using 181351 $\polP_3$ nodes at times 1, 1.5, 1.75, 2, 2.25 and 2.5.} \label{fig:rt_n50}
\end{figure}

\begin{figure}[H] \centering 
    \includegraphics[width=\xWidth \textwidth]{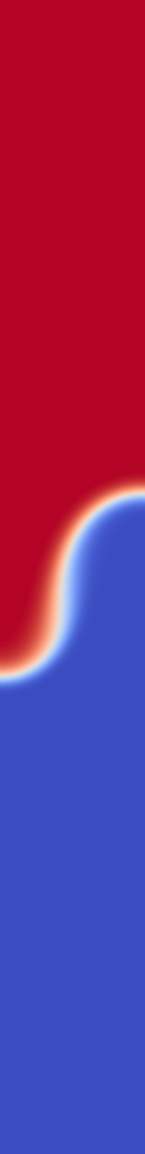}
    \includegraphics[width=\xWidth \textwidth]{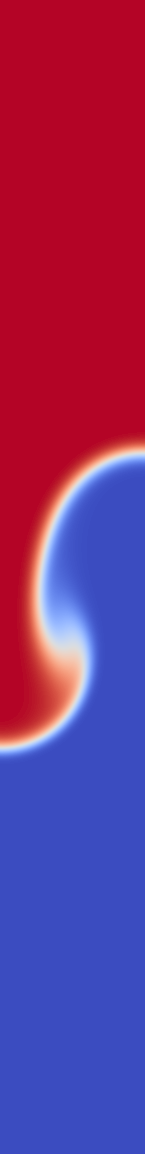}
    \includegraphics[width=\xWidth \textwidth]{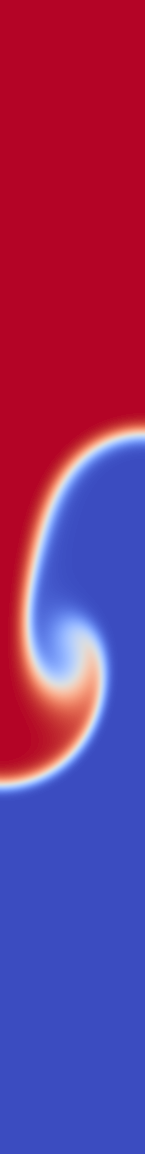}
    \includegraphics[width=\xWidth \textwidth]{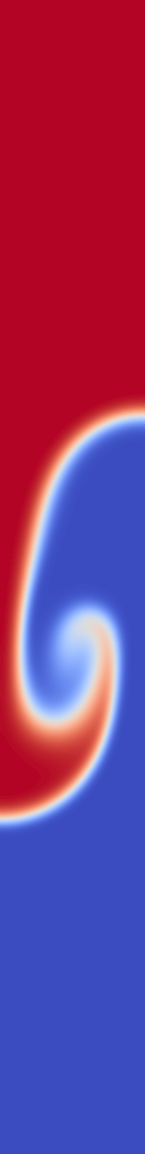}
    \includegraphics[width=\xWidth \textwidth]{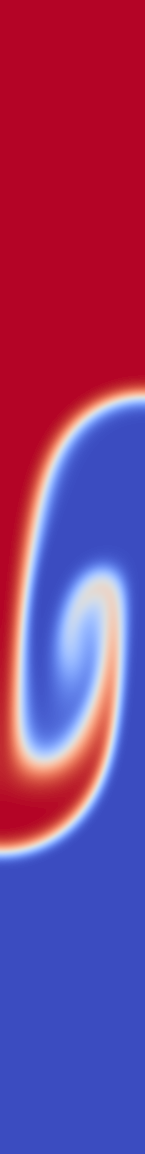}
    \includegraphics[width=\xWidth \textwidth]{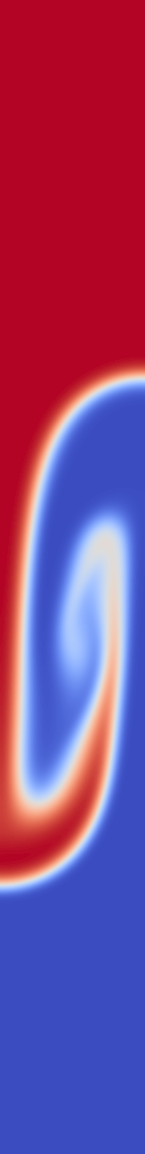}
    \includegraphics[width=\xWidthcolorbarr \textwidth]{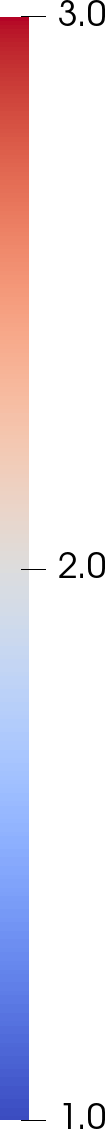}
  \caption{$Re = 5000$. Time evolution of $\rho$ using 722701 $\polP_3$ nodes field at times 1, 1.5, 1.75, 2, 2.25 and 2.5.} \label{fig:rt_n100}
\end{figure}

\begin{figure}[H] \centering 
    \includegraphics[width=\xWidth \textwidth]{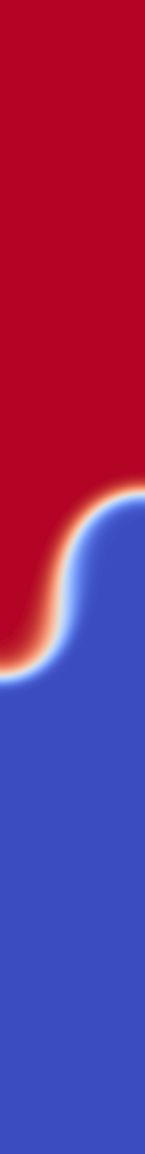}
    \includegraphics[width=\xWidth \textwidth]{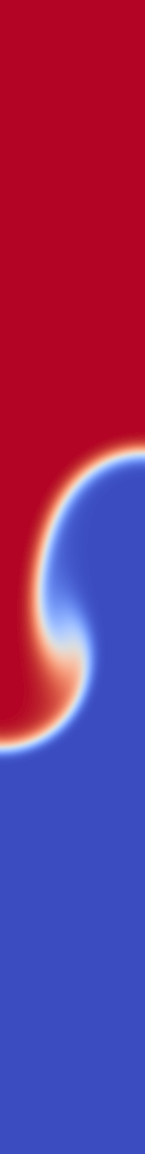}
    \includegraphics[width=\xWidth \textwidth]{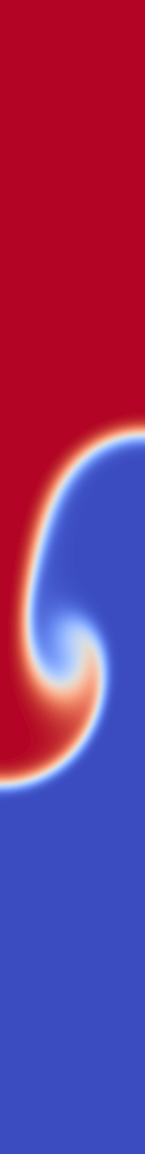}
    \includegraphics[width=\xWidth \textwidth]{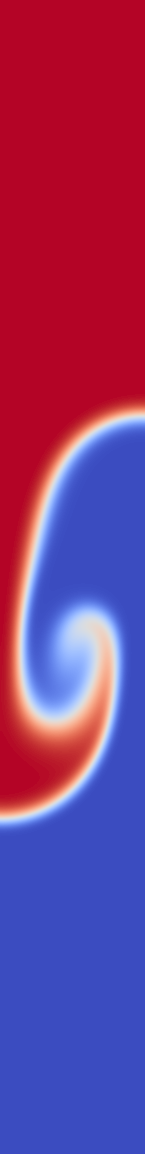}
    \includegraphics[width=\xWidth \textwidth]{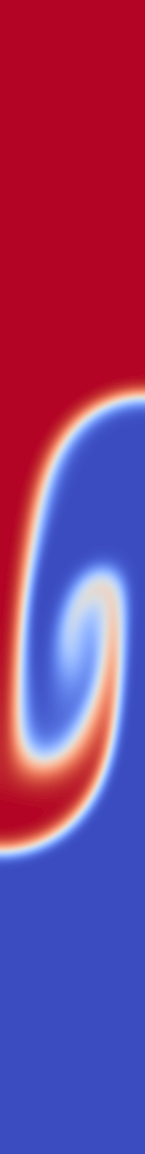}
    \includegraphics[width=\xWidth \textwidth]{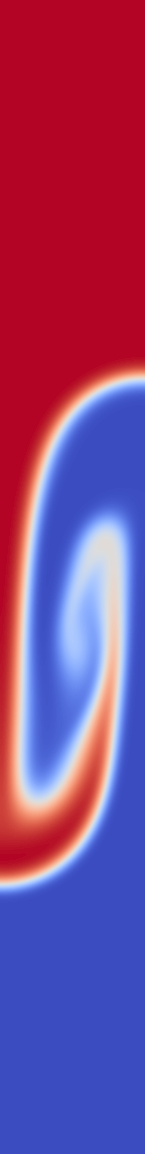}
    \includegraphics[width=\xWidthcolorbarr \textwidth]{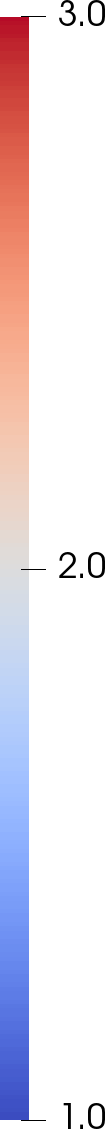}
  \caption{$Re = 5000$. Time evolution of $\rho$ using 722701 $\polP_3$ nodes field at times 1, 1.5, 1.75, 2, 2.25 and 2.5. Computed using first-order method from Section \ref{Sec:fem_approx}.} \label{fig:rt_n100_order1}
\end{figure}

\begin{figure}[H] \centering 
    \includegraphics[width=\xWidth \textwidth]{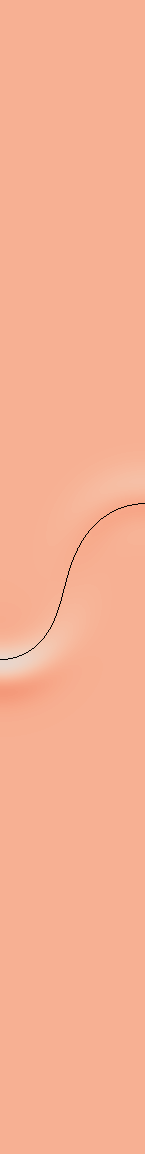}
    \includegraphics[width=\xWidth \textwidth]{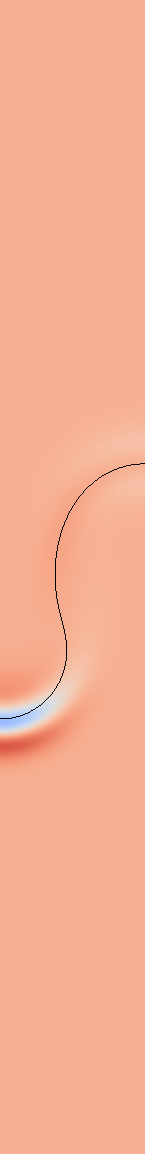}
    \includegraphics[width=\xWidth \textwidth]{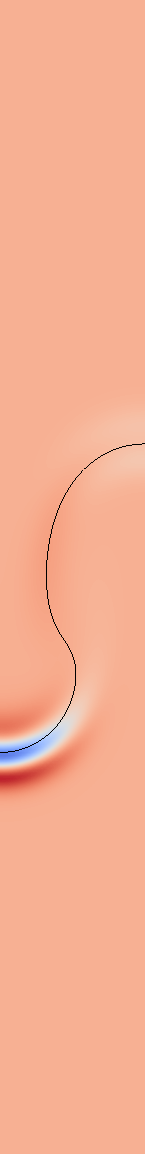}
    \includegraphics[width=\xWidth \textwidth]{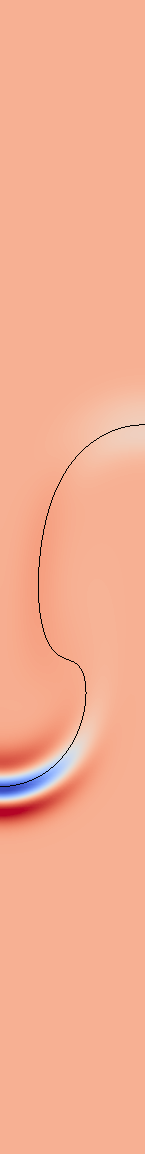}
    \includegraphics[width=\xWidth \textwidth]{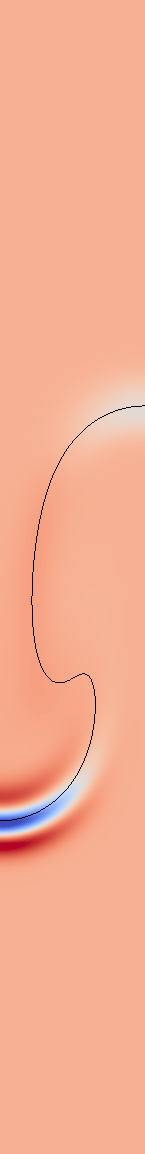}
    \includegraphics[width=\xWidth \textwidth]{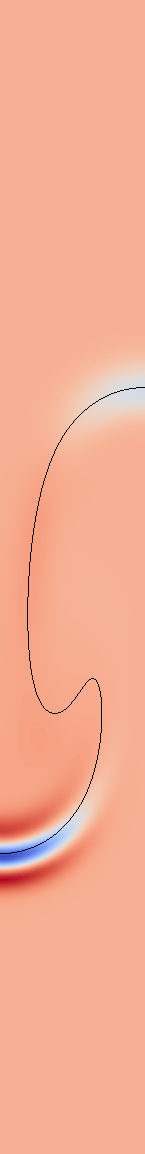}
    \includegraphics[width=\xWidthcolorbar \textwidth]{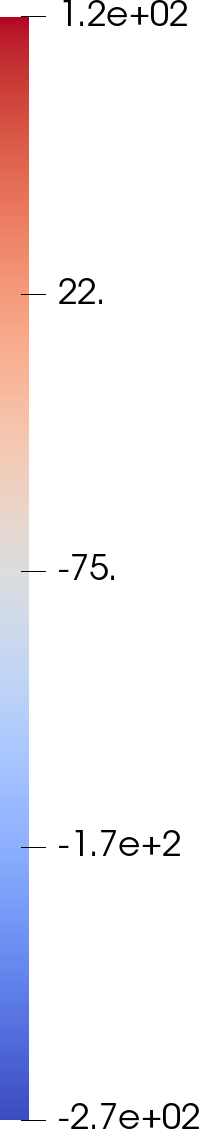}
  \caption{$Re = 5000$. Time evolution of $\rho_{ttt}$ using 45676 $\polP_3$ nodes at times 1, 1.5, 1.75, 2, 2.25 and 2.5. The black line corresponds to the contour line for $\rho = 2$ in Figure \ref{fig:rt_n25}.} \label{fig:rt_rhotl_n25}
\end{figure}

\begin{figure}[H] \centering 
    \includegraphics[width=\xWidth \textwidth]{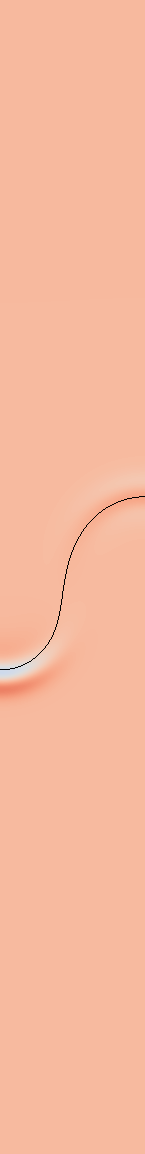}
    \includegraphics[width=\xWidth \textwidth]{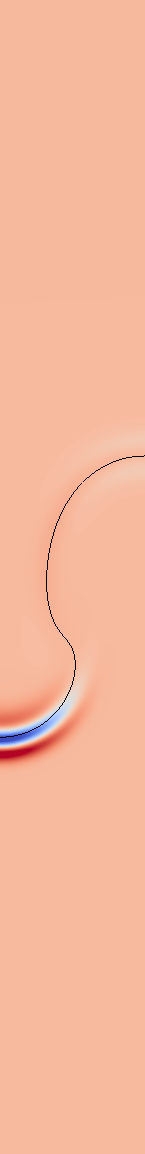}
    \includegraphics[width=\xWidth \textwidth]{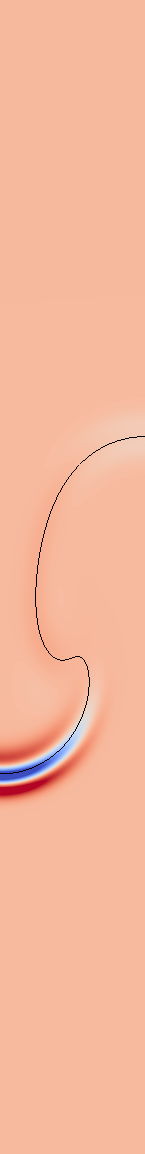}
    \includegraphics[width=\xWidth \textwidth]{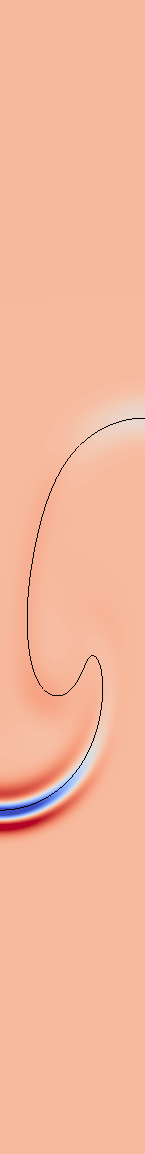}
    \includegraphics[width=\xWidth \textwidth]{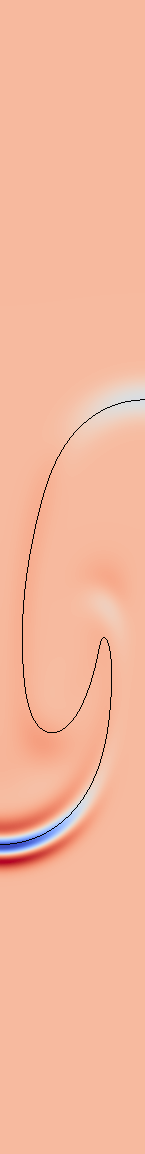}
    \includegraphics[width=\xWidth \textwidth]{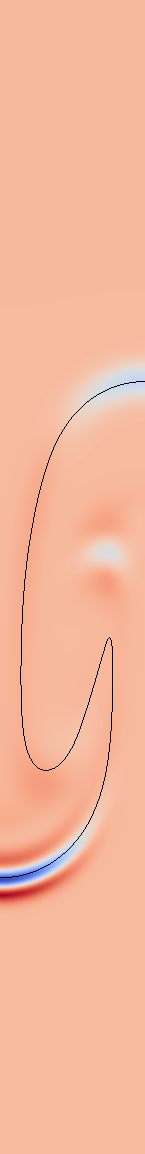}
    \includegraphics[width=\xWidthcolorbar \textwidth]{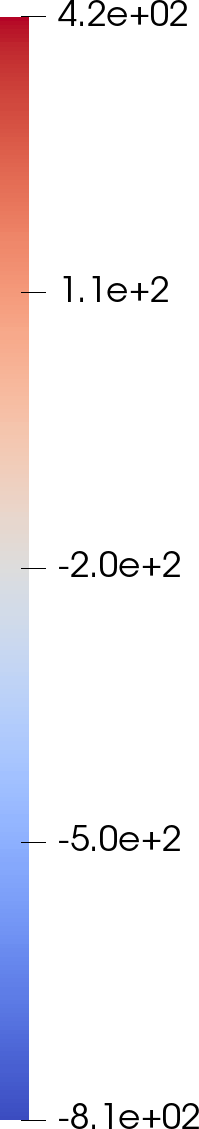}
  \caption{$Re = 5000$. Time evolution of $\rho_{ttt}$ using 181351 $\polP_3$ nodes at times 1, 1.5, 1.75, 2, 2.25 and 2.5. The black line corresponds to the contour line for $\rho = 2$ in Figure \ref{fig:rt_n50}.} \label{fig:rt_rhotl_n50}
\end{figure}

\begin{figure}[H] \centering 
    \includegraphics[width=\xWidth \textwidth]{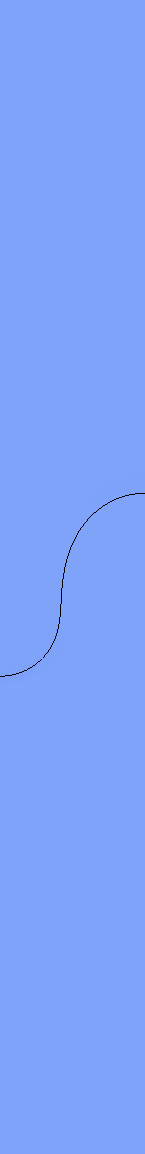}
    \includegraphics[width=\xWidth \textwidth]{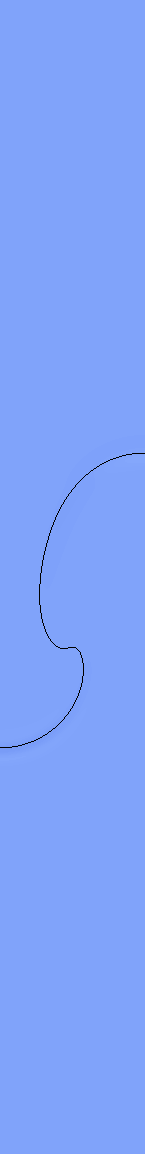}
    \includegraphics[width=\xWidth \textwidth]{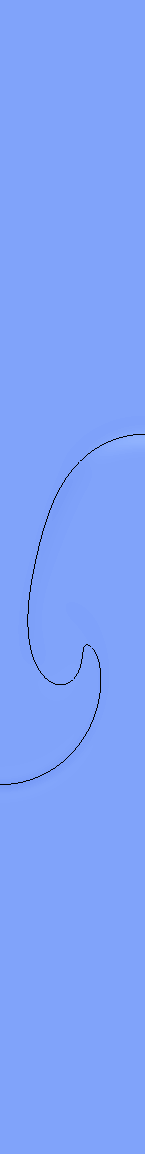}
    \includegraphics[width=\xWidth \textwidth]{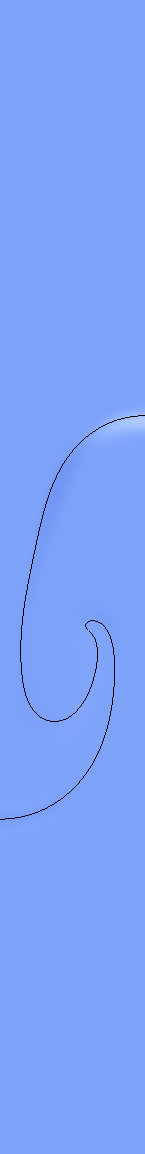}
    \includegraphics[width=\xWidth \textwidth]{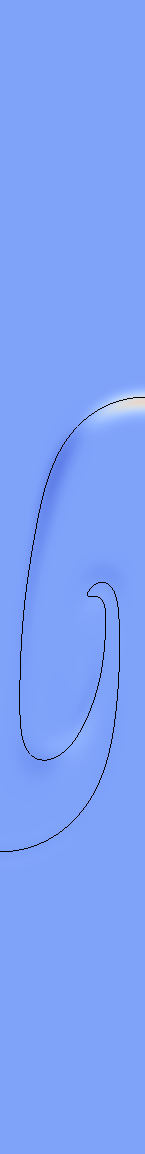}
    \includegraphics[width=\xWidth \textwidth]{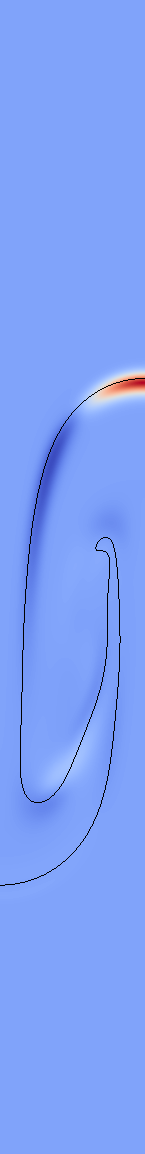}
    \includegraphics[width=\xWidthcolorbar \textwidth]{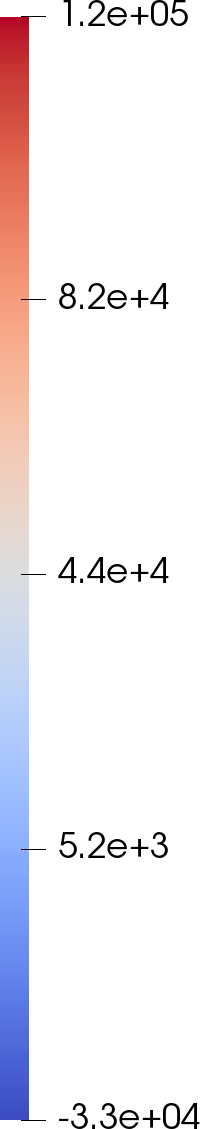}
  \caption{$Re = 5000$. Time evolution of $\rho_{ttt}$ using 722701 $\polP_3$ nodes at times 1, 1.5, 1.75, 2, 2.25 and 2.5. The black line corresponds to the contour line for $\rho = 2$ in Figure \ref{fig:rt_n100}.} \label{fig:rt_rhotl_n100}
\end{figure}



\begin{remark}
\commentA{Since $| \bu_0 | \approx 1$ for this problem, setting $\sigma_{h,1,2,3} = \nu_{h,1,2,3} = 0.5 h$ will lead to slightly more diffusive results than than $\sigma_{h,1,2,3} = \nu_{h,1,2,3} = 0.5 | \bu_0^n | h$. Qualitatively, the results are similar in the sense that the computed time derivatives grow without bound upon grid refinement.}
\end{remark}

\section{Conclusion and further work} \label{Sec:conclusion}A fourth-order accurate finite element method for variable density incompressible flow is proposed. The method uses artificial compressibility to impose the divergence-free constraint following the ideas presented by \cite{Guermond_Minev_2019} which we have extended to variable density. In space $\polP_3 \polP_3 \polP_2$ continuous finite elements are used and in time the Taylor series time-stepping method is used. If initial conditions for the time derivatives are available, the proposed method is high-order accurate using the penalty parameter $\lambda = 1$, which makes the condition number of the resulting linear systems suitable for large-scale applications. If these initial conditions are not available, a Richardson initialization procedure can be used which requires a larger $\lambda$.

We also illustrate that the Taylor series time-stepping method is not suitable when the regularity of the solution is lost by solving the Rayleigh-Taylor instability in 2D. The reason for this is that when the grid is refined, the time derivatives increase in magnitude without bound. Our current work is aimed at overcoming this problem, and the development of a reliable high-order finite element approximation for variable density flow will be reported in our future publication.

%
%

%

\section*{Acknowledgments}
\commentB{We thank the anonymous reviewers whose comments helped improve the quality of this manuscript. The computations were enabled by resources in project SNIC 2022/22-428 provided by the Swedish National Infrastructure for Computing (SNIC) at UPPMAX, partially funded by the Swedish Research Council through grant agreement no. 2018-05973. The first author was partly supported by the Center for Interdisciplinary Mathematics, Uppsala University.} 

\section{Appendix} \label{Sec:appendix}
In this section we present $L_2$ estimates for the first-order scheme from Section \ref{Sec:artificial_compressibility}. We denote $\| \cdot \|$ as the $L_2$-norm. To make the analysis simpler we assume that $\bu = 0 |_{\partial \Omega}$, which will make many boundary terms resulting from integration by parts disappear. Similarly, we set $\bef = 0$.

\subsection{First order scheme, density update}
We take the density update \eqref{eq:first_order1} and and discretize the flux skew-symmetrically. That gives us the following scheme
\begin{equation} \label{eq:gs_density}
\frac{\rho^{n+1} - \rho^n}{\tau} + \nabla \cdot \l( \rho^{n+1} \bu^n \r) - \frac{\rho^{n+1}}{2} \nabla \cdot \bu^n = 0.
\end{equation}
We now present an $L_2$ estimate for this scheme. The result and proof are standard and can be found elsewhere (for example \citep{Guermond_Salgado_2009}), but we recap it below for completeness.

\begin{proposition}
The density update \eqref{eq:gs_density} satisfies the following $L_2$-estimate
\begin{equation}
\begin{split}
\| \rho^{n+1}\|^2 - \| \rho^n \|^2 + \| \rho^{n+1} - \rho^n \|^2 = 0, \quad \forall n \geq 1.
\end{split}
\end{equation}
\end{proposition}

\begin{proof}
Multiplying \eqref{eq:gs_density} with $2 \tau \rho^{n+1}$, integrating over $\Omega$ and using $2a (a- b) = a^2 - b^2 + (a-b)^2$ yields

\begin{equation}
\begin{split}
& \| \rho^{n+1}\|^2 - \| \rho^n \|^2 + \| \rho^{n+1} - \rho^n \|^2 
+ & 2 \tau \l(\nabla \cdot \l( \rho^{n+1} \bu^n \r) , \rho^{n+1} \r) - \tau \l( \l( \rho^{n+1} \r)^2, \nabla \cdot \bu^n \r) = 0.
\end{split}
\end{equation}

For the advection term we have
\begin{equation}
\begin{split}
& 2 \l( \nabla \cdot \l( \rho^{n+1} \bu^n \r) , \rho^{n+1} \r) - \l(   \rho^{n+1} \nabla \cdot \bu^n , \rho^{n+1} \r)  \\
=& \l( \nabla \cdot \l( \rho^{n+1} \bu^n \r) , \rho^{n+1} \r) + \l( \nabla \cdot \l( \rho^{n+1} \bu^n \r) , \rho^{n+1} \r) - \l(   \rho^{n+1} \nabla \cdot \bu^n , \rho^{n+1} \r)   \\
=& \l( \nabla \cdot \l( \rho^{n+1} \bu^n \r) , \rho^{n+1} \r)  + \l(  \rho^{n+1} \nabla \cdot \bu^{n}  + \bu^n \cdot \nabla \rho^{n+1} , \rho^{n+1} \r) - \l(   \rho^{n+1} \nabla \cdot \bu^n , \rho^{n+1} \r)   \\
=& \l( \nabla \cdot \l( \rho^{n+1} \bu^n \r) , \rho^{n+1} \r) + \l(   \bu^n \cdot \nabla \rho^{n+1} , \rho^{n+1} \r) = 0,    \\
\end{split}
\end{equation}
where integration by parts was used in the last equality. We then have the desired estimate.





\end{proof}

\subsection{First order scheme, velocity form}



In this section, we present $L_2$ estimates of the momentum update. We take the first-order momentum update \eqref{eq:first_order2}-\eqref{eq:first_order3}, discretize the advection term skew-symmetrically and semi-implicitly and shift the density on the time derivative to $\rho^n$. To the best of the \commentB{authors' knowledge,} these modifications are necessary to obtain an $L_2$-estimate. The shifted density will unfortunately make high-order extensions (like those presented in this work) harder since it will limit the order of accuracy in time to first-order accuracy. There have been some attempts at overcoming this issue for BDF2 based time-stepping by \cite{Guermond_Salgado_2011}, but it remains an open question. The scheme that we consider is the following
\begin{equation} \label{eq:gs_mom}
\begin{split}
& \rho^n \frac{\bu^{n+1} - \bu^{n}}{\tau} + \rho^{n+1} \l( \bu^n \cdot \nabla \r) \bu^{n+1}  - \mu \commentB{ \Delta } \bu^{n+1} + \frac{\rho^{n+1}}{4} \l( \nabla \cdot \bu^n \r)  \bu^{n+1} + \nabla p^{n+1} = 0,
\end{split}
\end{equation}
\begin{equation} \label{eq:gs_press}
p^{n+1} = p^n - \lambda \nabla \cdot \bu^{n+1}.
\end{equation}

We follow a very similar proof technique as \cite{Guermond_Salgado_2009} did for variable density and as \cite{Chen_Layton_2019} did for artificial compressibility.

\begin{proposition}
Let $\sigma^n := \sqrt{\rho^n}$. The scheme for density \eqref{eq:gs_density}, velocity \eqref{eq:gs_mom} and pressure \eqref{eq:gs_press} satisfies the following $L_2$-estimate
\begin{equation} 
\begin{split}
\| \sigma^{n+1} \bu^{n+1} \|^2 - \| \sigma^n \bu^{n} \|^2 + \| \sigma^n \l( \bu^{n+1} - \bu^n \r) \|^2 + 2 \mu \tau \| \nabla \bu^{n+1} \|^2 \\ +  \frac{\tau }{\lambda} \l(  \| p^{n+1} \|^2 - \| p^n \|^2 + \| p^{n+1} - p^n \|^2 \r)   = 0,
\end{split} \quad \forall n \geq 1.
\end{equation}
\end{proposition}
\begin{proof}
We multiply \eqref{eq:gs_mom} with $2 \tau \bu^{n+1}$ and integrate over the domain. Integration by parts is performed on the diffusion term and the identity $2a (a- b) = a^2 - b^2 + (a-b)^2$ is applied \commentB{to} the time derivative term leading to
\begin{equation} \label{eq:mom_eq_mult2tau}
\begin{split}
\| \sigma^n \bu^{n+1} \|^2 - \| \sigma^n \bu^n \|^2 + \| \sigma^n \l( \bu^{n+1} - \bu^n \r) \|^2 + 2 \mu \tau \| \nabla \bu^{n+1} \|^2 \\ + \tau \int \rho^{n+1} \bu^{n} \cdot \nabla |\bu^{n+1} |^2
 + \tau/2 \int \rho^{n+1} ( \nabla \cdot \bu^n ) | \bu^{n+1} |^2  + 2 \tau (\nabla p^{n+1}, \bu^{n+1} ) = 0.
\end{split}
\end{equation}
Next, we take the density update \eqref{eq:gs_density} and multiply with $ \tau |\bu^{n+1} |^2$. Integrating over the domain and integrating by parts then gives
\begin{equation} \label{eq:dens_modified}
\begin{split}
\| \sigma^{n+1} \bu^{n+1} \|^2 - \| \sigma^n \bu^{n+1} \|^2  - \tau \int \rho^{n+1} \bu^n \cdot \nabla | \bu^{n+1} |^2 - \tau/2 \int \rho^{n+1} ( \nabla \cdot \bu^n ) |\bu^{n+1} |^2 = 0.
\end{split}
\end{equation}
Adding \eqref{eq:mom_eq_mult2tau} and \eqref{eq:dens_modified} and performing integration by parts on the pressure term yields
\begin{equation} \label{eq:mom_eq_trick}
\begin{split}
\| \sigma^{n+1} \bu^{n+1} \|^2 - \| \sigma^n \bu^{n} \|^2 + \| \sigma^n \l( \bu^{n+1} - \bu^n \r) \|^2 
+ 2 \mu \tau \| \nabla \bu^{n+1} \|^2 - 2 \tau \l( p^{n+1},  \nabla \cdot \bu^{n+1} \r) = 0.
\end{split}
\end{equation}

To deal with the pressure term we first multiply the pressure update \eqref{eq:gs_press} with $2 \tau p^{n+1}$, integrate over the domain and apply the identity $2a (a- b) = a^2 - b^2 + (a-b)^2$. Performing these steps gives
\begin{equation} \label{eq:press_eq_trick}
\begin{split}
\tau \l(  \| p^{n+1} \|^2 - \| p^n \|^2 + \| p^{n+1} - p^n \|^2 \r)  + 2 \tau \l( p^{n+1} , \lambda \nabla \cdot \bu^{n+1} \r) = 0.
\end{split}
\end{equation}

Finally, adding \eqref{eq:mom_eq_trick} and \eqref{eq:press_eq_trick} gives our desired $L_2$-estimate

\begin{equation*} 
\begin{split}
0 = &\| \sigma^{n+1} \bu^{n+1} \|^2 - \| \sigma^n \bu^{n} \|^2 + \| \sigma^n \l( \bu^{n+1} - \bu^n \r) \|^2 + 2 \mu \tau \| \nabla \bu^{n+1} \|^2 \\ & + \tau \l(  \| p^{n+1} \|^2 - \| p^n \|^2 + \| p^{n+1} - p^n \|^2 \r) - 2 \tau \l( p^{n+1},  (1 - \lambda) \nabla \cdot \bu^{n+1} \r) \\ \\
= & \| \sigma^{n+1} \bu^{n+1} \|^2 - \| \sigma^n \bu^{n} \|^2 + \| \sigma^n \l( \bu^{n+1} - \bu^n \r) \|^2 + 2 \mu \tau \| \nabla \bu^{n+1} \|^2 \\ & + \tau \l(  \| p^{n+1} \|^2 - \| p^n \|^2 + \| p^{n+1} - p^n \|^2 \r)  - 2 \tau \l( p^{n+1},  -(1 - \lambda) \frac{1}{\lambda} (p^{n+1} - p^n) \r)  \\ \\
= & \| \sigma^{n+1} \bu^{n+1} \|^2 - \| \sigma^n \bu^{n} \|^2 + \| \sigma^n \l( \bu^{n+1} - \bu^n \r) \|^2 + 2 \mu \tau \| \nabla \bu^{n+1} \|^2 \\ & + \tau \l(  \| p^{n+1} \|^2 - \| p^n \|^2 + \| p^{n+1} - p^n \|^2 \r)  +  \tau  \frac{1 - \lambda}{\lambda} \l(  \|p^{n+1}\|^2  - \| p^n\|^2 + \|p^{n+1} - p^n\|^2 \r)   \\ \\
= & \| \sigma^{n+1} \bu^{n+1} \|^2 - \| \sigma^n \bu^{n} \|^2 + \| \sigma^n \l( \bu^{n+1} - \bu^n \r) \|^2 + 2 \mu \tau \| \nabla \bu^{n+1} \|^2 \\ & +  \frac{\tau }{\lambda} \l(  \| p^{n+1} \|^2 - \| p^n \|^2 + \| p^{n+1} - p^n \|^2 \r) .
\end{split}
\end{equation*}
\phantom{invisible text}
\end{proof}

\newpage
\bibliographystyle{abbrvnat} 
\bibliography{ref}
\end{document}

\bibliographystyle{siam}
\bibliography{ref}
\end{document}